\documentclass[a4paper,12pt,leqno]{amsart}
\usepackage{latexsym}
\usepackage[all]{xy}

\usepackage{amsmath, amsfonts, amssymb,amscd,graphics,graphicx,color,eucal,setspace,xypic} 
\usepackage{comment}
\usepackage{mathtools}

\usepackage{graphicx}
\usepackage{stmaryrd}
\usepackage{bigdelim, multirow}%for underbrace
\usepackage[all]{xy}%for commutative diagram

\usepackage{mathdots}

\definecolor{gray}{gray}{0.7}
\definecolor{Gray}{gray}{0.3}

\usepackage{amscd}

\numberwithin{equation}{section}

\theoremstyle{break}
 \newtheorem{theorem}{Theorem}[section]
 \newtheorem{proposition}[theorem]{Proposition}
 \newtheorem{corollary}[theorem]{Corollary}
 \newtheorem{lemma}[theorem]{Lemma}

 \theoremstyle{definition}
 \newtheorem{definition}[theorem]{Definition}
 \newtheorem{remark}[theorem]{Remark}
 \newtheorem{example}[theorem]{Example}
 \newtheorem{question}[theorem]{Question}

\allowdisplaybreaks[4]

\def\C{\mathbb C}

\def\Q{\mathbb Q}
\def\Z{\mathbb Z}

\def\t{\mathfrak{t}}
\def\b{\mathfrak{b}}
\def\g{\mathfrak{g}}
\def\rank{n}
\def\m{r}
\def\Fl{Fl}
\def\perm{\mathfrak{S}}
\def\a{a}
\def\bb{b}

\def\RR{\mathcal R}
\def\L{\mathcal L}

\def\q{q}

\DeclareMathOperator{\diag}{diag}

\DeclareMathOperator{\pt}{pt}
\DeclareMathOperator{\Pet}{Pet}
\DeclareMathOperator{\Ad}{Ad}
\DeclareMathOperator{\Poin}{Poin}
\DeclareMathOperator{\Hom}{Hom}
\DeclareMathOperator{\Sym}{Sym}

\begin{document}

\title[Cohomology of intersection of Peterson variety and Richardson variety]{Toward cohomology rings of intersections of Peterson varieties and Richardson varieties}

\author[T. Horiguchi]{Tatsuya Horiguchi}
\address{National Institute of Technology, Ube College, 2-14-1, Tokiwadai, Ube, Yamaguchi, 755-8555, Japan}
\email{tatsuya.horiguchi0103@gmail.com}

%\date{\today}


\abstract 
Peterson varieties are subvarieties of flag varieties and their (equivariant) cohomology rings are given by Fukukawa--Harada--Masuda in type $A$ and soon later the author with Harada and Masuda gives an explicit presentation of the (equivariant) cohomology rings of Peterson varieties for arbitrary Lie types. 
In this note we study the (equivariant) cohomology ring of the intersections of Peterson variety with Schubert, opposite Schubert, and Richardson varieties in more general. 
By the work of Goldin--Mihalcea--Singh, the intersections of Peterson variety with Schubert varieties are naturally identified with smaller Peterson varieties, so the problem reduces to the problem for opposite Schubert intersections.
In this note we provide a technical statement for (equivariant) cohomology ring of a subvariety with some conditions of Peterson variety. 
By using the statement, we calculate the (equivariant) cohomology rings for some intersections of Peterson varieties with opposite Schubert varieties in type $A$. 
We also explicitly present the (equivariant) cohomology rings for some intersections of Peterson varieties with Richardson varieties in type $A$.

\endabstract

\maketitle

\tableofcontents

\section{Introduction}\label{intro}

The main result of this paper gives an explicit presentation of the (equivariant)
cohomology rings of the intersections between Peterson varieties and some Richardson varieties in terms of generators and relations. 
We briefly explain the background of our results.
Let $G$ be a semisimple linear algebraic group over $\C$ of rank $\rank$.
Fix a Borel subgroup $B$ of $G$ and a maximal torus $T$ in $B$.
We denote by $\Phi$ the associated root system. 
Peterson variety, denoted by $\Pet_\Phi$, is a subvariety of the flag variety $G/B$ which is appeared in the study of the quantum cohomology of flag varieties (\cite{Kos, Rie}).
It is known that $\Pet_\Phi$ is singular in general (\cite{IY, Kos}).
An explicit presentation of the cohomology ring\footnote{In this manuscript, the cohomology means the singular cohomology with rational coefficients.} of $\Pet_\Phi$ is given by \cite{FHM} when $\Phi$ is of type $A$. 
Soon later, \cite{HHM} gives an explicit presentation of the cohomology ring of $\Pet_\Phi$ for general Lie type, which is uniform across Lie types as follows. 
Take a set of simple roots $\Delta=\{\alpha_1,\ldots,\alpha_\rank \}$ of $\Phi$. 
Let $\{\varpi_1,\ldots,\varpi_\rank\}$ be the set of the fundamental weights. 
We denote by $\t^*_{\Q}$ the vector space over $\Q$ spanned by $\{\varpi_1,\ldots,\varpi_\rank \}$ (or equivalently spanned by $\{\alpha_1,\ldots,\alpha_\rank \}$). 
Then, the cohomology ring $H^*(\Pet_\Phi)$ is isomorphic to the quotient of the symmetric algebra $\Sym \t^*_\Q$ of the ideal generated by the quadratic forms $\alpha_i \varpi_i$ for $1 \leq i \leq \rank$: 
\begin{align} \label{eq:cohomologyPetIntro}
H^*(\Pet_\Phi) \cong \Sym \t^*_\Q/(\alpha_i \varpi_i \mid 1 \leq i \leq \rank).
\end{align}
Let $w_K$ denotes the longest element of the parabolic subgroup associated with $K \subset \Delta$. 
Consider the Schubert variety $X_{w_I}=\overline{Bw_IB/B}$ and the opposite Schubert variety $\Omega_{w_J}=\overline{B^{-}w_JB/B}$ for $I, J \subset \Delta$ where $B^-$ is the opposite Borel subgroup. 
We define $X_I$ and $\Omega_J$ by 
\begin{align*}
X_I \coloneqq X_{w_I} \cap \Pet_\Phi \ \textrm{and} \ \Omega_J \coloneqq \Omega_{w_J} \cap \Pet_\Phi.
\end{align*}
Note that $X_I=\Pet_\Phi$ whenever $I=\Delta$ and $\Omega_J=\Pet_\Phi$ whenever $J=\emptyset$ by the definition.
The homology classes of $\{X_I \}_{I \subset \Delta}$ forms a basis of the integral homology of $\Pet_\Phi$ (\cite{GMS, Pre13}). 
On the other hand, \cite{AHKZ} constructs an additive basis for the integral cohomology of $\Pet_\Phi$ which reflects $\{\Omega_J \}_{J \subset \Delta}$ in type $A$. 
Note that the cohomology basis coincides with a basis in \cite{HaTy11} which is the image of certain Schubert classes under the restriction map $H^*(G/B) \to H^*(\Pet_\Phi)$. 
To study the structure constants for this basis is called Peterson Schubert calculus (\cite{AHKZ, BaHa, Dre15, GoGo, GMS, HaTy11}).
In this paper, we study the cohomology rings for $X_I$ and $\Omega_J$, and more generally 
\begin{align*}
X_I^J \coloneqq X_I \cap \Omega_J \ \ \ \textrm{for} \ J \subset I \subset \Delta
\end{align*}
which is the intersection between the Peterson variety $\Pet_\Phi$ and the Richardson variety $X_{w_I} \cap \Omega_{w_J}$.
By \cite{GMS} Schubert intersections are identified with smaller Peterson varieties. 
This fact with \eqref{eq:cohomologyPetIntro} yields that one can calculate the cohomology ring for $X_I$.
More concretely, it is described as  
\begin{align} \label{eq:cohomologyXIIntro}
H^*(X_I) \cong \Sym \t^*_\Q/(\alpha_k \varpi_k \mid k \in [\rank] \ \textrm{with} \ \alpha_k \in I) +(\varpi_i \mid i \in [\rank] \ \textrm{with} \ \alpha_i \in \Delta \setminus I).
\end{align}
Hence, the problem of Richardson intersections in our setting reduces to the problem for opposite Schubert intersections.

The approach in \cite{HHM} to obtain \eqref{eq:cohomologyPetIntro} is localization technique.
In fact, \cite{HHM} gives an explicit presentation of the $S$-equivariant cohomology ring $H_S^*(\Pet_\Phi)$ where $S$ is a one-dimensional subtorus of the maximal torus $T$ which preserves $\Pet_\Phi$ introduced by Harada and Tymoczko (\cite{HaTy11, HaTy17}).
One of the key points in the proof of \eqref{eq:cohomologyPetIntro} is that we need the following two formulae of the $S$-fixed point set $(\Pet_\Phi)^S$ and the Poincar\'e polynomial $\Poin(\Pet_\Phi,\q)$ given by \cite{HaTy17, Pre13}.
\begin{enumerate}
\item 
The set of $S$-fixed points $(\Pet_\Phi)^S$ is given by 
\begin{align*}
(\Pet_\Phi)^S=\{w_K \in W \mid K \subset \Delta \}. 
\end{align*}
\item The Poincar\'e polynomial of $\Pet_\Phi$ is equal to 
\begin{align*}
\Poin(\Pet_\Phi,\q)=(1+q^2)^\rank.
\end{align*}
\end{enumerate}
From this point we provide the following technical statement.

\begin{theorem} \label{theorem:Intro1}
Let $I$ and $J$ be two subsets of the simple roots $\Delta=\{\alpha_1,\ldots,\alpha_\rank\}$ with $J \subset I$.
Let $Y_I^J$ be a subvariety of the Peterson variety $\Pet_\Phi$ satisfying the following two conditions.
\begin{enumerate}
\item The $S$-action on $\Pet_{\Phi}$ preserves $Y_I^J$ and the set of $S$-fixed points is 
\begin{align*}
(Y_I^J)^S=\{w_K \in W \mid K \subset \Delta, J \subset K \subset I \}.
\end{align*}
\item The Poincar\'e polynomial of $Y_I^J$ is equal to 
\begin{align*}
\Poin(Y_I^J,\q)=(1+q^2)^{|I|-|J|}.
\end{align*}
\end{enumerate}
Then, there is an isomorphism of graded $\Q$-algebras
\begin{align} \label{eq:cohomologyYIJIntro}
H^*(Y_I^J) \cong \Sym \t^*_\Q/\L_I^J,
\end{align}
where the ideal $\L_I^J$ is generated by the following three types of elements:
\begin{align*}
&\alpha_k \varpi_k \ \ \ \textrm{for} \ k \in [\rank] \ \textrm{with} \ \alpha_k \in I \setminus J; \\
&\alpha_j  \ \ \ \ \ \ \textrm{for} \ j \in [\rank] \ \textrm{with} \ \alpha_j \in J; \\
&\varpi_i \ \ \ \ \ \ \textrm{for} \ i \in [\rank] \ \textrm{with} \ \alpha_i \in \Delta \setminus I.
\end{align*}
Here, we use the notation $[n] \coloneqq \{1,2,\ldots,n\}$.
\end{theorem}

We remark that the second condition~(2) for $Y_I^J$ in Theorem~\ref{theorem:Intro1} implies that the odd degree cohomology groups of $Y_I^J$ vanish. 
By using the localization technique, we prove Theorem~\ref{theorem:Intro1}.
More precisely, we give an explicit presentation of the $S$-equivariant cohomology ring of $Y_I^J$ (Theorem~\ref{theorem:psiIJiso}), which yields the isomorphism \eqref{eq:cohomologyYIJIntro}.

Note that $X_I$ satisfies the conditions~(1) and (2) in Theorem~\ref{theorem:Intro1} when $J=\emptyset$, which yields \eqref{eq:cohomologyXIIntro}.
On the other hand, we can compute the cohomology ring of the specail case of $\Omega_J$ in type $A$ by using Theorem~\ref{theorem:Intro1}.
One can easily check that the $S$-action on $\Pet_\Phi$ preserves $\Omega_J$ and  
\begin{align*}
(\Omega_J)^S=\{w_K \in W \mid K \subset \Delta, K \supset J \}
\end{align*}
for any $J \subset \Delta$ in type $A_n$. 
However, it is complicated to compute the Poincar\'e polynomial of $\Omega_J$ in general. 
If $J$ is of the form $\{\alpha_1,\alpha_2,\ldots,\alpha_m\}$ for some $m$ in type $A_n$, then we prove that $\Poin(\Omega_{\{\alpha_1,\alpha_2,\ldots,\alpha_m\}},\q)=(1+\q^2)^{n-m}$.
Therefore, $\Omega_{\{\alpha_1,\alpha_2,\ldots,\alpha_m\}}$ of type $A_n$ satisfies the conditions~(1) and (2) in Theorem~\ref{theorem:Intro1} when $I=\Delta$ and $J=\{\alpha_1,\alpha_2,\ldots,\alpha_m\}$, which derives an explicit presentation of the cohomology ring of $\Omega_{\{\alpha_1,\alpha_2,\ldots,\alpha_m\}}$. 

\begin{theorem} \label{theorem:Intro3}
Let $\Delta=\{\alpha_1,\ldots,\alpha_\rank\}$ be the set of the simple roots in type $A_n$. 
Suppose that $J=\{\alpha_1,\alpha_2,\ldots,\alpha_m\}$ for some $m$ with $1 \leq m \leq \rank$.
Let $\Omega_J$ be the intersection between the opposite Schubert variety $\Omega_{w_J}$ and the Peterson variety $\Pet_\Phi$ in type $A_n$.
Then, there is an isomorphism of graded $\Q$-algebras
\begin{align*}
H^*(\Omega_{\{\alpha_1,\alpha_2,\ldots,\alpha_m\}}) \cong \Sym \t^*_\Q/(\alpha_k \varpi_k \mid m+1 \leq k \leq \rank) +(\alpha_j \mid 1 \leq j \leq m).
\end{align*}
\end{theorem}

In more general, we have the following result for the special Richardson intersections in type $A$ by using Theorem~\ref{theorem:Intro1}.

\begin{theorem} \label{theorem:Intro4}
Let $\Delta=\{\alpha_1,\ldots,\alpha_\rank\}$ be the set of the simple roots in type $A_n$. 
Suppose that $I=\{\alpha_r,\alpha_{r+1},\ldots,\alpha_n\}$ and $J=\{\alpha_r,\alpha_{r+1},\ldots,\alpha_m\}$ for some $r$ and $m$ with $1 \leq r \leq m \leq \rank$.
Let $X_I^J$ be the intersection between the Richardson variety $X_{w_I} \cap \Omega_{w_J}$ and the Peterson variety $\Pet_\Phi$ in type $A_n$.
Then, there is an isomorphism of graded $\Q$-algebras
\begin{align*}
H^*(X_I^J) \cong \Sym \t^*_\Q/(\alpha_k \varpi_k \mid m+1 \leq k \leq n) +(\alpha_j \mid r \leq j \leq m)+(\varpi_i \mid 1 \leq i \leq r-1).
\end{align*}
\end{theorem}

We expect that a candidate for $Y_I^J$ in Theorem~\ref{theorem:Intro1} is the intersection $X_I^J$ between the Peterson variety $\Pet_\Phi$ and the Richardson variety $X_{w_I} \cap \Omega_{w_J}$, but it is \emph{not} true in general.
In fact, one can verify that if $I=\{\alpha_1, \alpha_2, \alpha_3\}$ and $J=\{\alpha_1, \alpha_3\}$ in type $A_3$, then $X_I^J$ does \emph{not} satisfy the condition~(2) in Theorem~\ref{theorem:Intro1} (see Example~\ref{ex:non-irreducible case}). 
However, in this case, every irreducible component of $X_{\{\alpha_1, \alpha_2, \alpha_3\}}^{\{\alpha_1, \alpha_3\}}$ does satisfy the conditions~(1) and (2) in Theorem~\ref{theorem:Intro1}.
Hence, we can ask whether a candidate for $Y_I^J$ in Theorem~\ref{theorem:Intro1} is some irreduclble component of $X_I^J$ (Questions~\ref{question:OmegaJ} and \ref{question:XIJ}).

The paper is organized as follows. After reviewing the definition of Peterson variety
and the result of \cite{HHM}, we state our technical statement (Theorem~\ref{theorem:Intro1}) in Section~\ref{section:Peterson}. 
In Section~\ref{section:Relation} we derive the fundamental relations in the cohomology ring of $Y_I^J$.
We then prove Theorem~\ref{theorem:Intro1} in Section~\ref{section:Proof}.
In Section~\ref{section:Intersections of Peterson variety and Schubert varieties} we explain the work of Goldin--Mihalcea--Singh in \cite{GMS} to describe the cohomology rings of the intersections of Peterson variety and Schubert varieties \eqref{eq:cohomologyXIIntro}. 
We also verify the presentation \eqref{eq:cohomologyXIIntro} by using Theorem~\ref{theorem:Intro1}.
Consider the opposite situation in Section~\ref{section:Intersections of Peterson variety and opposite Schubert varieties} and we give an explicit presentation of the cohomology ring for the special case of $\Omega_J$ in type $A$ (Theorem~\ref{theorem:Intro3}). 
Finally, we give an explicit presentation of the cohomology ring for the special case of $X_I^J$ in type $A$ (Theorem~\ref{theorem:Intro4}) in Section~\ref{section:Intersections of Peterson variety and Richardson varieties}. 
Note that Sections~\ref{section:Intersections of Peterson variety and opposite Schubert varieties} and \ref{section:Intersections of Peterson variety and Richardson varieties} have questions about a candidate for $Y_I^J$ in Theorem~\ref{theorem:Intro1}.

\bigskip
\noindent \textbf{Acknowledgements.}  
The author is grateful to Mikiya Masuda, Hiraku Abe, Hideya Kuwata, and Haozhi Zeng for valuable discussions. 
The author is also grateful to the anonymous referee for many perceptive comments and suggestions.
This research was partly supported by Osaka City University Advanced Mathematical Institute (MEXT Joint Usage/Research Center on Mathematics and Theoretical Physics).
The author is supported in part by JSPS Grant-in-Aid for Young Scientists: 19K14508.

\bigskip

\section{Peterson varieties} \label{section:Peterson}

In this section we review the definition of Peterson varieties for all Lie types and an explicit presentation of the cohomology ring of Peterson varieties given in \cite{HHM}.

Let $G$ be a semisimple linear algebraic group over $\C$ of rank $\rank$.
Fix a Borel subgroup $B$ of $G$ and a maximal torus $T$ in $B$.
We denote by $\t \subset \b \subset \g$ the Lie algebras of $T \subset B \subset G$. 
Let $\Phi$ be the associated root system and $\g_\alpha$ denotes the root space for a root $\alpha$.
Fix a simple system $\Delta:=\{\alpha_1,\ldots,\alpha_\rank\}$. 
Let $N$ be a regular nilpotent element of $\g$, namely $N$ is nilpotent and its $G$-orbit of the
adjoint action has the largest possible dimension. 
Then the Peterson variety is defined by 
\begin{align*}
\Pet_{\Phi}:=\{gB \in G/B \mid \Ad(g^{-1})(N) \in \b \oplus \bigoplus_{i=1}^\rank \g_{-\alpha_i} \},
\end{align*}
which is a subvariety of the flag variety $G/B$. 
Any regular nilpotent element is $G$-conjugate to the regular nilpotent element of the form $\sum_{i=1}^\rank E_{\alpha_i}$ where $E_{\alpha_i}$ is a basis of the root space $\g_{\alpha_i}$ for each $1 \leq i \leq \rank$. 
Since an arbitrary Peterson variety associated to an arbitrary regular nilpotent element $N$ is isomorphic to the specific Peterson variety associated to the regular nilpotent element $\sum_{i=1}^\rank E_{\alpha_i}$, we may assume that $N=\sum_{i=1}^\rank E_{\alpha_i}$ (cf. \cite[Lemma~5.1]{HaTy17}). 
It is known that $\Pet_{\Phi}$ is irreducible and its complex dimension is equal to $\rank$ (\cite{AFZ, Pre13, Pre18}).
Also, $\Pet_{\Phi}$ has a singular point in general (\cite{IY, Kos}).
The following result gives a formula of the Poincar\'e polynomial of $\Pet_\Phi$, which includes the fact that the odd Betti numbers of $\Pet_\Phi$ vanish.

\begin{theorem} $($\cite[Corollary~4.13]{Pre13}$)$ \label{theorem:PoincarePolynomialPet}
The Poincar\'e polynomial of $\Pet_\Phi$ is given by 
\begin{align*}
\Poin(\Pet_\Phi,\q)=(1+q^2)^\rank.
\end{align*} 
\end{theorem}

The maximal torus $T$ acts on the flag variety $G/B$. 
This $T$-action does not preserve the Peterson variety $\Pet_{\Phi}$.
However, it is known that there exists a one-dimensional subtorus $S$ of $T$ such that the $S$-action on $G/B$ by restriction  preserves $\Pet_{\Phi}$. 
This $S$-action is introduced by Harada and Tymoczko (\cite{HaTy11, HaTy17}) and we explain the definition of $S$ below. 
Let $\phi: T \to (\C^*)^\rank$ be a homomorphism defined by $\phi(g):=(\alpha_1(g),\ldots,\alpha_\rank(g))$ for $g \in T$. 
Define $S$ as the connected component of the identity in $\phi^{-1}(\{(c,\ldots,c) \mid c \in \C^* \})$. 
It then follows from the definition of $S$ that the one-dimensional subtorus $S$ preserve $\Pet_{\Phi}$ (\cite[Lemma~5.1-(3)]{HaTy17}).

We next explain the $S$-fixed points of the Peterson variety $\Pet_{\Phi}$. 
As is well-known, the set of the $T$-fixed points $(G/B)^T$ of the flag variety $G/B$ can be identified with the Weyl group $W$.
It follows from \cite[Lemma~5.1-(3)]{HaTy17} that the set of $S$-fixed points $(G/B)^S$ of $G/B$ is equal to $(G/B)^T$ and hence we have 
$$
(\Pet_{\Phi})^S = (G/B)^T \cap \Pet_{\Phi}.
$$
Thus, one can regard the $S$-fixed points $(\Pet_{\Phi})^S$ of $\Pet_\Phi$ as a subset of the Weyl group $W \cong (G/B)^T$. 
The $S$-fixed points $(\Pet_{\Phi})^S \subset W$ is described as follows. 
First, we may regard the set of simple roots $\Delta$ as the corresponding Dynkin diagram.
Every subset $K \subset \Delta$ can be decomposed as $K=K_1 \times K_2 \times \cdots \times K_m$ into connected Dynkin diagrams $K_i$'s.
Then we define $w_K \in W$ by 
\begin{align*}
w_K = w_0^{K_1} \times w_0^{K_2} \times \cdots \times w_0^{K_m}
\end{align*}
where $w_0^{K_i}$ is the longest element of the parabolic subgroup $W_{K_i}$ which is generated by simple reflections $s_k$ for $\alpha_k \in K_i$. 
Note that $w_K$ is the identity element whenever $K$ is the empty set.

\begin{lemma} $($\cite[Proposition~5.8]{HaTy17}$)$ \label{lemma:PetFixedPoints}
The set of $S$-fixed points $(\Pet_\Phi)^S$ is given by $\{w_K \in W \mid K \subset \Delta \}$. 
\end{lemma}

We now turn to an explicit presentation of the cohomology ring of the Peterson variety $\Pet_{\Phi}$ given by \cite{HHM}. (For type $A$, it was given in \cite{FHM}.)
Here and below, we may identify the character group $\Hom(T,\C^*)$ of $T$ with a lattice $\t^*_{\Z}$ through differential at the identity element of $T$.
We also denote the symmetric algebra of $\t^*_{\Q}:=\t^*_{\Z} \otimes_{\Z} \Q$ by 
\begin{align*}
\RR = \Sym \t^*_{\Q}.
\end{align*}
Let $\{\varpi_1,\ldots,\varpi_\rank\}$ be the set of the fundamental weights.
Then we have that
\begin{align*}
\RR=\Q[\varpi_1,\ldots,\varpi_\rank]=\Q[\alpha_1,\ldots,\alpha_\rank]
\end{align*}
since $\{\varpi_1,\ldots,\varpi_\rank\}$ or $\{\alpha_1,\ldots,\alpha_\rank\}$ is a $\Q$-basis of $\t^*_{\Q}$.
Here, we briefly remind the reader of the relationship between the simple roots $\{\alpha_1,\ldots,\alpha_\rank \}$, the fundamental weights $\{\varpi_1,\ldots,\varpi_\rank \}$, the Killing form, and the Cartan matrix. We refer the reader to \cite{Hum72} for the details.
For $x, y \in \g$, the Killing form $\kappa$ is defined by $\kappa(x,y)=\mbox{trace} (\mbox{ad}(x) \, \mbox{ad}(y))$ which is bilinear and symmetric form on $\g$. 
As is well-known, the Killing form $\kappa$ is non-degenerate since $\g$ is semisimple. 
It is also known that the restriction of $\kappa$ to $\t$ is non-degenerate, so this allows us to identify $\t$ with its dual $\t^*$. 
We may transfer this form to $\t^*$, denoted by $( \ , \ )$. 
One has $(\alpha,\beta) \in \Q$ for any $\alpha, \beta \in \Phi$, so we obtain a non-degenerate form on $\t^*_\Q$. 
We can also see that the form on $\t^*_\Q$ is positive definite. 
For $\alpha, \beta \in \Phi$, we write $\langle \alpha, \beta \rangle$ for the number $\frac{2(\alpha,\beta)}{(\beta,\beta)} \in \Z$. 
The Cartan matrix $(c_{ij})_{i,j \in [\rank]}$ is defined by $c_{ij}=\langle \alpha_i, \alpha_j \rangle$ and we have $\alpha_i = \sum_{j=1}^{\rank} c_{ij} \varpi_j = \sum_{j=1}^{\rank} \langle \alpha_i, \alpha_j \rangle \varpi_j$ by the definition of the fundamental weights. 

For a homomorphism $\chi: B \rightarrow \C^*$, we can consider the $B$-action on $G \times \C_{\chi}$ given by $b \cdot (g, z):=(gb^{-1},\chi(b)z)$ for $b \in B$ and $(g,z) \in G \times \C_{\chi}$. Note that $\C_{\chi}$ denotes the one-dimensional $B$-module via $\chi: B \rightarrow \C^*$.
The quotient space $G \times_B \C_{\chi}$ is the line bundle over the flag variety $G/B$. 
Since any homomorphism $\alpha: T \rightarrow \C^*$ extends to a homomorphism $\tilde\alpha: B \rightarrow \C^*$, we can assign the line bundle $L_{\alpha}:=G \times_B \C_{\tilde \alpha}$ to each $\alpha \in \t^*_{\Z}$.
By taking the first Chern class $c_1(L_{\alpha}^*)$ of the dual line bundle $L_{\alpha}^*$, we have the following ring homomorphism
\begin{align} \label{eq:homomorphism_flag}
\RR \rightarrow H^*(G/B); \ \ \ \alpha \mapsto c_1(L_{\alpha}^*) 
\end{align} 
which doubles the grading on $\RR$.
By Borel's result (\cite{Bor}), the map above is surjective and the kernel is the ideal generated by the $W$-invariants in $\RR$ with zero constant term. 
Since the restriction map $H^*(G/B) \rightarrow H^*(\Pet_\Phi)$ is surjective (\cite[Theorem~3.5]{Dre15}), one can see that the composition of the map \eqref{eq:homomorphism_flag} and the restriction map
\begin{align} \label{eq:homomorphism_Pet}
\varphi: \RR \rightarrow H^*(G/B) \rightarrow H^*(\Pet_\Phi); \ \ \ \alpha \mapsto c_1(L_{\alpha}^*)
\end{align} 
is also surjective.
Here and below, by abuse of notation we also denote by the same symbol $L_{\alpha}$ the restriction $L_{\alpha}|_{\Pet_{\Phi}}$. 

\begin{theorem} $($\cite[Theorem~4.1]{HHM}$)$ \label{theorem:CohomologyPet}
The kernel of the map $\varphi$ is the ideal generated by $\alpha_1\varpi_1, \ldots, \alpha_\rank\varpi_\rank$.
In particular, we obtain the following explicit presentation 
\begin{align} \label{eq:Cohomology_Pet}
H^*(\Pet_\Phi) \cong \RR/(\alpha_i \varpi_i \mid i \in [\rank])
\end{align}
where $[\rank]:=\{1,2,\ldots,\rank\}$. 
\end{theorem}

The proof of the presentation above in \cite{HHM} is based on the localization technique, which is frequently used for giving an explicit presentation of the (equivariant) cohomology ring.
We now briefly recall the argument.
We first note that the restriction map 
\begin{align*}
H_S^*(\Pet_\Phi) \rightarrow H_S^*\big((\Pet_\Phi)^S\big)=\bigoplus_{w \in (\Pet_\Phi)^S} H_S^*(\pt) 
\end{align*}
is injective since the odd degree cohomology groups of $\Pet_\Phi$ vanish (\cite[Theorem~4.10]{Pre13}).
By the injectivity of the restriction map above, one can calculate the relations in $H_S^*(\Pet_\Phi)$ because we know the fixed points data (Lemma~\ref{lemma:PetFixedPoints}).
From this, one can construct the homomorphism from a quotient ring (which is an equivariant version of the right hand side in \eqref{eq:Cohomology_Pet}) to the equivariant cohomology $H^*_S(\Pet_\Phi)$.
Since we know a formula of the Poincar\'e polynomial for $\Pet_\Phi$ (Theorem~\ref{theorem:PoincarePolynomialPet}), one can compare the Hilbert series of the quotient ring and $H^*_S(\Pet_\Phi)$. 
As a result, we obtain an explicit presentation of $H^*_S(\Pet_\Phi)$, which yields our presentation \eqref{eq:Cohomology_Pet} as desired. 
To summarize, we need the following two data:
\begin{enumerate}
\item fixed points $(\Pet_\Phi)^S$; 
\item Poincar\'e polynomial $\Poin(\Pet_\Phi,\q)$.
\end{enumerate}
From this point we prepare the following technical statement.

\begin{theorem} \label{theorem:Cohomology1}
Let $I$ and $J$ be two subsets of the simple roots $\Delta=\{\alpha_1,\ldots,\alpha_\rank\}$ with $J \subset I$.
Let $Y_I^J$ be a subvariety of the Peterson variety $\Pet_\Phi$ satisfying the following two conditions.
\begin{enumerate}
\item The $S$-action on $\Pet_{\Phi}$ preserves $Y_I^J$ and the set of $S$-fixed points is 
\begin{align*}
(Y_I^J)^S=\{w_K \in W \mid K \subset \Delta, J \subset K \subset I \}.
\end{align*}
\item The Poincar\'e polynomial of $Y_I^J$ is equal to 
\begin{align*}
\Poin(Y_I^J,\q)=(1+q^2)^{|I|-|J|}.
\end{align*}
\end{enumerate}
Then, the map $\varphi$ in \eqref{eq:homomorphism_Pet} composed with the restriction map $H^*(\Pet_\Phi) \rightarrow H^*(Y_I^J)$, denoted by
\begin{align*}
\varphi_I^J: \RR \xrightarrow{\varphi} H^*(\Pet_\Phi) \rightarrow H^*(Y_I^J)
\end{align*}
is surjective. (Note that this is equivalent to saying that the restriction map $H^*(\Pet_\Phi) \rightarrow H^*(Y_I^J)$ is surjective). 
Moreover, the kernel of $\varphi_I^J$ is the ideal $\L_I^J$ generated by the following three types of elements:
\begin{align*}
&\alpha_k \varpi_k \ \ \ \textrm{for} \ k \in [\rank] \ \textrm{with} \ \alpha_k \in I \setminus J; \\
&\alpha_j  \ \ \ \ \ \ \textrm{for} \ j \in [\rank] \ \textrm{with} \ \alpha_j \in J; \\
&\varpi_i \ \ \ \ \ \ \textrm{for} \ i \in [\rank] \ \textrm{with} \ \alpha_i \in \Delta \setminus I. 
\end{align*}
In particular, we obtain the following isomorphism of graded $\Q$-algebras: 
\begin{align*}
H^*(Y_I^J) \cong \RR/\L_I^J.
\end{align*}
\end{theorem}

\begin{remark} \label{remark:YIJ odd degree}
The second condition~(2) for $Y_I^J$ in Theorem~\ref{theorem:Cohomology1} implies that the odd degree cohomology groups of $Y_I^J$ vanish. 
We will use the injectivity of the restriction map 
\begin{align*} 
H_S^*(Y_I^J) \hookrightarrow H_S^*\big((Y_I^J)^S\big)=\bigoplus_{w \in (Y_I^J)^S} H_S^*(\pt) 
\end{align*} 
for the proof of Theorem~\ref{theorem:Cohomology1}.
\end{remark}

Before we prove Theorem~\ref{theorem:Cohomology1}, we see some application examples of Theorem~\ref{theorem:Cohomology1}.

\begin{example} \label{example:Cohomology1}
\begin{enumerate}
\item[(i)] Take $I=\Delta$ and $J=\emptyset$ in Theorem~\ref{theorem:Cohomology1}. Then, we can apply Theorem~\ref{theorem:Cohomology1} to the Peterson variety $Y_I^J=\Pet_\Phi$. 
In fact, we have $(\Pet_\Phi)^S=\{w_K \in W \mid K \subset \Delta \}$ by Lemma~\ref{lemma:PetFixedPoints} and the Poincar\'e polynomial of $\Pet_\Phi$ is given by $\Poin(\Pet_\Phi,\q)=(1+q^2)^\rank$ from Theorem~\ref{theorem:PoincarePolynomialPet}. 
\item[(ii)] In general, we consider the case when $J=\emptyset$ but $I$ is an arbitrary subset of $\Delta$ in Theorem~\ref{theorem:Cohomology1}. 
Let $X_I \coloneqq X_{w_I} \cap \Pet_\Phi$ be the intersection of the Schubert variety $X_{w_I}$ and the Peterson variety $\Pet_\Phi$ where we recall that the Schubert variety $X_w$ is the closure of the $B$-orbit of $wB$ in $G/B$ for $w \in W$.
By \cite{GMS} their intersections are identified with smaller Peterson varieties, so this with Theorem~\ref{theorem:CohomologyPet} yields that 
\begin{align*} 
H^*(X_I) \cong \RR/(\alpha_k \varpi_k \mid k \in [\rank] \ \textrm{with} \ \alpha_k \in I) +(\varpi_i \mid i \in [\rank] \ \textrm{with} \ \alpha_i \in \Delta \setminus I).
\end{align*}
On the other hand, the $S$-action on $\Pet_\Phi$ preserves $X_I$. We verify that $(X_I)^S =\{w_K \mid K \subset I \}$ and $\Poin(X_I,q)=(1+q^2)^{|I|}$ (Lemma~\ref{lemma:condition(1)XI} and Proposition~\ref{proposition:condition(2)XI}). 
Therefore, we can also apply Theorem~\ref{theorem:Cohomology1} for presenting the cohomology ring $H^*(X_I)$ as above.
The details will be discussed in Section~\ref{section:Intersections of Peterson variety and Schubert varieties}.
\item[(iii)] Consider the case when $I=\Delta$ but $J$ is an arbitrary subset of $\Delta$ in Theorem~\ref{theorem:Cohomology1}, which is the opposite situation to (ii). 
We denote by $\Omega_w \ (w \in W)$ the opposite Schubert variety which is the closure of the $B^-$-orbit of $wB$ in $G/B$. 
Here, $B^-$ denotes the Borel subgroup defined by $w_0Bw_0$ where $w_0$ is the longest element of $W$.
Consider the intersections $\Omega_J \coloneqq \Omega_{w_J} \cap \Pet_\Phi$. 
(Note that these intersections have been studied in \cite{AHKZ} for type $A$ which are related to a geometric interpretation for Peterson Schubert calculus.)
We will prove in Section~\ref{section:Intersections of Peterson variety and opposite Schubert varieties} that if $J$ is of the form $\{\alpha_1,\alpha_2,\ldots,\alpha_m\}$ for some $m$ in type $A_{n-1}$, then $\Omega_J$ satisfies the conditions~(1) and~(2) in Theorem~\ref{theorem:Cohomology1}. 
Thus, we obtain an explicit presentation of $H^*(\Omega_{\{\alpha_1,\alpha_2,\ldots,\alpha_m\}})$ in type $A_{n-1}$. 
\item[(iv)] In more general, we consider $X_I^J \coloneqq X_I \cap \Omega_J$, which is the intersection between the Richardson variety $X_{w_I} \cap \Omega_{w_J}$ and the Peterson variety $\Pet_\Phi$. 
If $I=\{\alpha_r,\alpha_{r+1},\ldots,\alpha_{\rank-1}\}$ and $J=\{\alpha_r,\alpha_{r+1},\ldots,\alpha_m\}$ for some $r$ and $m$ in type $A_{n-1}$, then we will see an explicit presentation of the cohomology ring of $X_I^J$ in Section~\ref{section:Intersections of Peterson variety and Richardson varieties}. 
\end{enumerate}
\end{example}

\bigskip

\section{Relations in $H^*_S(Y_I^J)$} \label{section:Relation}

In this section we first recall the fundamental relations in $H^*_S(\Pet_\Phi)$ given by \cite{HHM}.
Then we describe relations in $H^*_S(Y_I^J)$, which are modified from the fundamental relations in $H^*_S(\Pet_\Phi)$.

For a homomorphism $\alpha: T \rightarrow \C^*$, let $\C_\alpha$ be the complex $1$-dimensional representation of $T$ defined by $g \cdot z =\alpha(g) z$ for all $g \in T$ and $z \in \C_\alpha$. 
It is known that the correspondence 
\begin{align*}
\t^*_{\Q}=\t^*_{\Z} \otimes_{\Z} \Q \rightarrow H^2_T(\pt); \ \alpha \mapsto c_1^T(\C_\alpha)
\end{align*}
gives an isomorphism where $c_1^T(\C_{\alpha})$ denotes the $T$-equivariant first Chern class of $\C_{\alpha}$.
For simplicity, we write 
$$
\alpha^T=c_1^T(\C_{\alpha}) \in H^2_T(\pt)
$$ 
for any $\alpha \in \t^*_{\Q}$. 
In particular, we have the following identification 
\begin{align*}
H^*_T(\pt) = \Q[\alpha_1^T,\ldots,\alpha_\rank^T]
\end{align*}
since the set of simple roots $\{\alpha_1,\ldots,\alpha_\rank\}$ forms a $\Q$-basis of $\t^*_{\Q}$.
Let $\rho$ be the character of $S$ defined by the following composition map 
\begin{align*}
S \hookrightarrow T \xrightarrow{\prod_{i=1}^\rank{\alpha_i}} (\C^*)^\rank \ \ \ \textrm{and} \ \ \ \{(c,\ldots,c) \mid c \in \C^* \} \rightarrow \C^*; (c,\ldots,c) \mapsto c.
\end{align*}
Note that the image of the former map above is in the subgroup $\{(c,\ldots,c) \mid c \in \C^* \}$ by the definition of $S$.
Let $t$ be the $S$-equivariant first Chern class of the associated line bundle $\C_\rho$, namely 
\begin{align*}
t \coloneqq c_1^S(\C_\rho) \in H^2_S(\pt).
\end{align*}
With this notation in place we have 
\begin{align*}
H^*_S(\pt) = \Q[t].
\end{align*}
The inclusion $S \hookrightarrow T$ induces a homomorphism $\pi: H^*_T(\pt) \rightarrow H^*_S(\pt)$ and one can easily see from the definition of $\rho$ above that 
\begin{align} \label{eq:alpha to t}
\pi(\alpha_i^T)=t \ \ \ \textrm{for all} \ i \in [\rank].
\end{align}
Recall that the map $H^*_T(\pt) \rightarrow H^*_T(G/B)$ induced from the Borel fibration allows us to an $H^*_T(\pt)$-module for $H^*_T(G/B)$. 
Here we note that the homomorphism $H^*_T(\pt) \rightarrow H^*_T(G/B)$ is injective since $H^{odd}(G/B)=0$. 
One can extend \eqref{eq:homomorphism_flag} to the equivariant version as follows:
\begin{align*} 
\RR[\alpha_1^T,\ldots,\alpha_\rank^T] \rightarrow H^*_T(G/B); \ \ \ \alpha \mapsto c_1^T(L_{\alpha}^*) \ \textrm{for} \ \alpha \in \RR \ \textrm{and} \ \alpha_i^T \mapsto \alpha_i^T, 
\end{align*} 
where $c_1^T(L_{\alpha}^*)$ denotes the $T$-equivariant first Chern class of $L_{\alpha}^*$.
Similarly, we can construct the $\Q[t]$-algebra homomorphism
\begin{align*} 
\RR[t] \rightarrow H^*_S(G/B); \ \ \ \alpha \mapsto c_1^S(L_{\alpha}^*) \ \textrm{and} \ t \mapsto t. 
\end{align*} 
The map above composed with the restriction map $H^*_S(G/B) \mapsto H^*_S(\Pet_\Phi)$ is also the equivariant version of $\varphi$ in \eqref{eq:homomorphism_Pet}.
We denote it by 
\begin{align} \label{eq:equivariant_homomorphism_Pet}
\tilde\varphi: \RR[t] \rightarrow H^*_S(G/B) \rightarrow H^*_S(\Pet_\Phi).
\end{align} 
It follows from \cite[Theorem~3.5]{Dre15} that $\tilde\varphi$ is surjective.

We now explain the fundamental relations in $H^*_S(\Pet_\Phi)$ by \cite{HHM}.
Consider the following commutative diagram:
\begin{equation} \label{eq:CD}
\begin{CD}
H^*_T(G/B)@>{\iota_1}>> H^*_T \big((G/B)^T \big) = \displaystyle \bigoplus_{w\in W} \Q[\alpha_1^T,\dots,\alpha_\rank^T]\\
@V{}VV @VV{}V\\
H^*_S(\Pet_\Phi)@>{\iota_2}>> H^*_S \big((\Pet_\Phi)^S \big) = \displaystyle \bigoplus_{w\in (\Pet_\Phi)^S} \Q[t]
\end{CD}
\end{equation}
where all the maps are induced from the inclusion maps on underlying spaces.
Since $G/B$ and $\Pet_\Phi$ have no odd degree cohomology, both $\iota_1$ and $\iota_2$ are injective. 
Thus, we may regard $H^*_T(G/B)$ and $H^*_S(\Pet_\Phi)$ as a subset of $\bigoplus_{w\in W} \Q[\alpha_1^T,\dots,\alpha_\rank^T]$ and $\bigoplus_{w\in (\Pet_\Phi)^S} \Q[t]$, respectively. 
For $f \in H^*_T(G/B) \subset \bigoplus_{w\in W} \Q[\alpha_1^T,\dots,\alpha_\rank^T]$, the $w$-th component of $f$ is denoted by $f|_w$.
Similarly, the $w$-th component of $f \in H^*_S(\Pet_\Phi) \subset \bigoplus_{w\in (\Pet_\Phi)^S} \Q[t]$ is also denoted by $f|_w$.
The following lemma is straightforward (cf. \cite[Lemma~5.2]{AHMMS}).

\begin{lemma} \label{lemma:c1TLalpha}
For any $\alpha \in \t^*_\Q$, the $w$-th component of $c_1^T(L_\alpha) \in H^2_T(G/B)$ is  
\begin{align*}
c_1^T(L_\alpha)|_w=w(\alpha^T) \ \ \  \textrm{in} \ H^2_T(\pt).
\end{align*}
\end{lemma}

For $w \in W$, let $\sigma_w \in H^*(G/B)$ be the Schubert class, which is the Poincar\'e dual of  the opposite Schubert variety $\Omega_w$ in $G/B$ (see Example~\ref{example:Cohomology1}-(iii) for the definition of $\Omega_w$). 
We also denote by $\tilde\sigma_w \in H^*_T(G/B)$ the corresponding $T$-equivariant Schubert class. 
Let $\tilde p_w \in H^*_S(\Pet_\Phi)$ be the image of the $T$-equivariant Schubert class $\tilde \sigma_w$ under the restriction map $H^*_T(G/B) \rightarrow H^*_S(\Pet_\Phi)$.
Similarly, we denote by $p_w \in H^*(\Pet_\Phi)$ the image of the (ordinary) Schubert class $\sigma_w$ under the (ordinary) restriction map $H^*(G/B) \rightarrow H^*(\Pet_\Phi)$.

\begin{proposition} $($\cite[Proposition~3.4]{HHM}$)$ \label{proposition:HHM}
Let $(c_{ij})_{i,j \in [\rank]}$ be the Cartan matrix.
Then we have 
\begin{align} \label{eq:fundamental relations in HSPet}
\big(\sum_{j=1}^\rank c_{ij} \tilde p_{s_j}-2t \big) \tilde p_{s_i}=0 \ \textrm{in} \ H^*_S(\Pet_\Phi)
\end{align}
for all $i \in [\rank]$.
\end{proposition}

We remark that the equations \eqref{eq:fundamental relations in HSPet} for all $i \in [n]$ give the fundamental relations in $H^*_S(\Pet_\Phi)$ by \cite[Theorem~4.1]{HHM}.
In order to see generators of the kernel of $\tilde\varphi$ in \eqref{eq:equivariant_homomorphism_Pet}, we need to know an element of $\RR[t]$ whose image under the map $\tilde\varphi$ is $\tilde p_{s_i}$. 
For this purpose, we prepare the following lemma.

\begin{lemma}  \label{lemma:psi_formula}
Let $(d_{ij})_{i,j \in [\rank]}$ be the inverse matrix of the Cartan matrix.
For any $i \in [\rank]$, we have 
\begin{align*}
\tilde p_{s_i} = c_1^S(L_{\varpi_i}^*) + \big( \sum_{j=1}^\rank d_{ij} \big) t \ \ \ \textrm{in} \ H^*_S(\Pet_\Phi).
\end{align*}
\end{lemma}

\begin{proof}
By the surjectivity of $\tilde\varphi$ in \eqref{eq:equivariant_homomorphism_Pet}, we can write
\begin{align*}
\tilde p_{s_i} = \sum_{k=1}^\rank a_k c_1^S(L_{\varpi_k}^*) + b \, t \ \ \ \textrm{in} \ H^*_S(\Pet_\Phi)
\end{align*}
for some $a_1, \ldots, a_\rank, b \in \Q$.
Taking the image of the both sides above under the forgetful map $ H^*_S(\Pet_\Phi) \rightarrow H^*(\Pet_\Phi)$, we have
\begin{align} \label{eq:proof psi1}
p_{s_i} = \sum_{k=1}^\rank a_k c_1(L_{\varpi_k}^*) \ \ \ \textrm{in} \ H^*(\Pet_\Phi).
\end{align}
It is known that $\sigma_{s_k}=c_1(L_{\varpi_k}^*)$ in $H^*(G/B)$ (cf. \cite{BGG}), so one has $p_{s_k}=c_1(L_{\varpi_k}^*)$ in $H^*(\Pet_\Phi)$ which implies that $p_{s_i} = \sum_{k=1}^\rank a_k p_{s_k}$ in $H^*(\Pet_\Phi)$. 
Since $p_{s_1}, \ldots, p_{s_\rank}$ are linearly independent by \cite[Theorem~3.5]{Dre15}, we have $a_i=1$ and $a_k=0$ if $k \neq i$. 
Thus, \eqref{eq:proof psi1} is written as 
\begin{align*} 
\tilde p_{s_i} = c_1^S(L_{\varpi_i}^*) + b \, t \ \ \ \textrm{in} \ H^*_S(\Pet_\Phi).
\end{align*}
Recall that the simple reflection $s_i$ is an $S$-fixed point of $\Pet_\Phi$ from Lemma~\ref{lemma:PetFixedPoints}.
Comparing the $s_i$-th component of both sides, one has  
\begin{align} \label{eq:proof psi2}
\tilde p_{s_i}|_{s_i} = c_1^S(L_{\varpi_i}^*)|_{s_i} + b \, t \ \ \ \textrm{in} \ H^*_S(\pt)=\Q[t].
\end{align}
As is well-known, $\tilde \sigma_{s_i}|_{s_i}=\alpha_i^T$ in $H^*_T(\pt)$ (cf. \cite{Bil}).
Thus, the left hand side of \eqref{eq:proof psi2} is equal to $\tilde p_{s_i}|_{s_i}=t$ by \eqref{eq:alpha to t}. 
On the other hand, Lemma~\ref{lemma:c1TLalpha} leads us to 
\begin{align*}
c_1^T(L_{\varpi_i}^*)|_{s_i}=-s_i(\varpi_i^T)=-(\varpi_i^T-\alpha_i^T)=\alpha_i^T- \sum_{j=1}^\rank d_{ij} \alpha_j^T \ \ \ \textrm{in} \ H^2_T(\pt).
\end{align*}
From this and \eqref{eq:alpha to t} again, we write \eqref{eq:proof psi2} as 
\begin{align*}
t=t - \big( \sum_{j=1}^\rank d_{ij} \big) t + b \, t \ \ \ \textrm{in} \ \Q[t].
\end{align*}
Hence, we obtain $b=\sum_{j=1}^\rank d_{ij}$ as desired.
\end{proof}

The following proposition is a translation of Proposition~\ref{proposition:HHM} by Lemma~\ref{lemma:psi_formula}.
We leave the proof of the proposition below to the reader.

\begin{proposition} \label{proposition:HHM2}
It holds that 
\begin{align*}
\big(c_1^S(L_{\alpha_i}^*) - t \big) \big(c_1^S(L_{\varpi_i}^*)+( \sum_{j=1}^\rank d_{ij}) t \big)=\big(\sum_{j=1}^\rank c_{ij} \tilde p_{s_j}-2t \big) \tilde p_{s_i}=0 \ \ \ \textrm{in} \ H^*_S(\Pet_\Phi)
\end{align*}
for all $i \in [\rank]$ where $(d_{ij})_{i,j \in [\rank]}$ is the inverse matrix of the Cartan matrix $(c_{ij})_{i,j \in [\rank]}$.
\end{proposition}

By Proposition~\ref{proposition:HHM2} we can rewrite \cite[Theorem~4.1]{HHM} as follows.

\begin{theorem} \label{theorem:EquivariantCohomologyPet}
The kernel of the map $\tilde\varphi$ in \eqref{eq:equivariant_homomorphism_Pet} is the ideal generated by 
\begin{align*}
\big( \alpha_i - t \big) \big(\varpi_i+( \sum_{j=1}^\rank d_{ij}) t \big)
\end{align*}
for all $i \in [\rank]$.
Here, $(d_{ij})_{i,j \in [\rank]}$ denotes the inverse matrix of the Cartan matrix.
In particular, we obtain 
\begin{align*}
H^*_S(\Pet_\Phi) \cong \RR[t]/\big( (\alpha_i - t)(\varpi_i+(\sum_{j=1}^\rank d_{ij}) t) \mid i \in [\rank] \big).
\end{align*}
\end{theorem}

We now see relations in $H^*_S(Y_I^J)$ where $Y_I^J$ is a subvariety of $\Pet_\Phi$ in Theorem~\ref{theorem:Cohomology1}.
In what follows, by abuse of notation again we denote by the same symbol $L_{\alpha}$ the restriction $L_{\alpha}|_{Y_I^J}$ of the line bundle $L_{\alpha}=L_{\alpha}|_{\Pet_\Phi}$ over $\Pet_\Phi$.

\begin{proposition} \label{proposition:relationsYIJ}
Let $Y_I^J$ be a subvariety of the Peterson variety $\Pet_\Phi$ satisfying the  conditions~(1) and~(2) in Theorem~\ref{theorem:Cohomology1}.
Then, we obtain
\begin{align*}
&\big(c_1^S(L_{\alpha_k}^*) - t \big) \big(c_1^S(L_{\varpi_k}^*)+( \sum_{\ell=1}^\rank d_{k\ell}) t \big)=0 \ \ \ \textrm{for} \ k \in [\rank] \ \textrm{with} \ \alpha_k \in I \setminus J; \\
&c_1^S(L_{\alpha_j}^*) - t=0  \hspace{135pt} \textrm{for} \ j \in [\rank] \ \textrm{with} \ \alpha_j \in J; \\
&c_1^S(L_{\varpi_i}^*)+( \sum_{j=1}^\rank d_{ij}) t=0 \hspace{90pt} \textrm{for} \ i \in [\rank] \ \textrm{with} \ \alpha_i \in \Delta \setminus I 
\end{align*}
in $H^*_S(Y_I^J)$.
\end{proposition}

\begin{proof}
By the restriction map $H^*_S(\Pet_\Phi) \rightarrow H^*_S(Y_I^J)$, the first relation follows from Proposition~\ref{proposition:HHM2} by restricting to $H^*_S(Y_I^J)$. 
We show the second and the third relations.
As explained in Remark~\ref{remark:YIJ odd degree}, we use the injectivity of the restriction map 
\begin{align*} 
H_S^*(Y_I^J) \hookrightarrow H_S^*\big((Y_I^J)^S\big)=\bigoplus_{w \in (Y_I^J)^S} H_S^*(\pt). 
\end{align*} 
In order to prove the second relation, we take a simple root $\alpha_j \in J$.
Then it suffices to show that 
\begin{align} \label{eq:proof lemma second relation}
c_1^S(L_{\alpha_j}^*)|_{w_K} =t  \ \ \ \textrm{for all} \ K \subset \Delta \ \textrm{with} \ J \subset K \subset I 
\end{align}
by the condition~(1) for $Y_I^J$.
By Lemma~\ref{lemma:c1TLalpha} one has
\begin{align*} 
c_1^T(L_{\alpha_j}^*)|_{w_K}=-w_K(\alpha_j^T) \ \ \ \textrm{in} \ H^2_T(\pt).
\end{align*}
Since $\alpha_j \in J$ and $J \subset K$, we have $\alpha_j \in K$.
Noting that $w_K$ is the longest element associated with the simple system $K$, one see that $w_K(K)=-K\coloneqq\{-\alpha_k \mid \alpha_k \in K \}$ in $\t^*_\Q$. 
This implies that $c_1^T(L_{\alpha_j}^*)|_{w_K}=\alpha_k^T$ in $H^2_T(\pt)$ for some $\alpha_k \in K$, which yields \eqref{eq:proof lemma second relation} from \eqref{eq:alpha to t}.

Finally, we show the third relation. 
To do that, it is enough to prove that for $i \in [\rank]$ with $\alpha_i \in \Delta \setminus I$, the image of $\tilde p_{s_i}$ under the restriction map $H^*_S(\Pet_\Phi) \rightarrow H^*_S(Y_I^J)$ vanishes by Lemma~\ref{lemma:psi_formula}. 
In other words, it suffices to show that
\begin{align} \label{eq:proof lemma third relation}
\tilde p_{s_i}|_{w_K} =0  \ \ \ \textrm{for all} \ K \subset \Delta \ \textrm{with} \ J \subset K \subset I 
\end{align}
where we denote by the same symbol $\tilde p_{s_i}$ the image of $\tilde p_{s_i}$ under the restriction map $H^*_S(\Pet_\Phi) \rightarrow H^*_S(Y_I^J)$.
Since $\alpha_i \notin I$ and $K \subset I$, one has $\alpha_i \notin K$ which implies that $s_i \not\leq w_K$ in Bruhat order.
As is well-known, the $T$-equivariant Schubert classes $\tilde\sigma_w$ satisfy an upper-triangularity property which means that $\tilde\sigma_w|_v=0$ if $v \not\geq w$ (cf. \cite{Bil}).
Thus, we have $\tilde\sigma_{s_i}|_{w_K}=0$ which implies \eqref{eq:proof lemma third relation} as desired.
\end{proof}

\bigskip

\section{Proof of Theorem~\ref{theorem:Cohomology1}} \label{section:Proof}

In this section we devote to give a proof of Theorem~\ref{theorem:Cohomology1}. 
An idea of the proof is based on \cite{AHHM}.
Throughout this section, $Y_I^J$ denotes a subvariety of the Peterson variety $\Pet_\Phi$ satisfying the conditions~(1) and~(2) in Theorem~\ref{theorem:Cohomology1}.

As usual, we consider the map $\tilde\varphi$ in \eqref{eq:equivariant_homomorphism_Pet} composed with the restriction map $H^*_S(\Pet_\Phi) \rightarrow H^*_S(Y_I^J)$ as follows:
\begin{align} \label{eq:tilde phiIJ}
\tilde \varphi_I^J: \RR[t] \xrightarrow{\tilde\varphi} H^*_S(\Pet_\Phi) \rightarrow H^*_S(Y_I^J).
\end{align} 
Define $\tilde \L_I^J$ by the ideal of $\RR[t]$ generated by the following three types of elements:
\begin{align}
&\theta_k \coloneqq \big( \alpha_k -t \big) \big( \varpi_k + (\sum_{\ell=1}^\rank d_{k\ell}) t \big) \ \ \ \textrm{for} \ k \in [\rank] \ \textrm{with} \ \alpha_k \in I \setminus J; \label{eq:theta} \\
&\xi_j \coloneqq \alpha_j -t  \ \ \ \ \ \ \textrm{for} \ j \in [\rank] \ \textrm{with} \ \alpha_j \in J; \label{eq:xi} \\
&\nu_i \coloneqq \varpi_i + (\sum_{j=1}^\rank d_{ij}) t \ \ \ \ \ \ \textrm{for} \ i \in [\rank] \ \textrm{with} \ \alpha_i \in \Delta \setminus I, \label{eq:nu}
\end{align}
where $(d_{ij})_{i,j \in [\rank]}$ is the inverse matrix of the Cartan matrix. 
Then, by Proposition~\ref{proposition:relationsYIJ}
the map $\tilde \varphi_I^J$ in \eqref{eq:tilde phiIJ} induces the following map
\begin{align} \label{eq:psiIJ}
\psi_I^J: \RR[t]/\tilde \L_I^J \rightarrow H^*_S(Y_I^J).
\end{align}
We will show that $\psi_I^J$ is an isomorphism.
To do that, we first compare the Hilbert series of the both sides.  
Since the odd degree cohomology groups of $Y_I^J$ vanish (Remark~\ref{remark:YIJ odd degree}), there exists an isomorpshism as $H^*_S(\pt)$-modules 
\begin{align} \label{eq:free_HS(YIJ)}
H^*_S(Y_I^J) \cong H^*_S(\pt) \otimes H^*(Y_I^J) = \Q[t] \otimes H^*(Y_I^J).
\end{align}
The isomorphism above yields the Hilbert series of $H^*_S(Y_I^J)$ as follows:
\begin{align} \label{eq:HilbertSeriesHSYIJ}
F(H^*_S(Y_I^J),q)=\frac{1}{1-q^2} \Poin(Y_I^J,q)=\frac{(1+q^2)^{|I|-|J|}}{1-q^2}.
\end{align}
The last equality follows from the condition~(2) for $Y_I^J$.
Next we compute the Hilbert series of $\RR[t]/\tilde \L_I^J$ by using commutative algebra tools. 

\begin{definition}\label{definition:regular sequence}
Let $R$ be a graded commutative algebra over $\Q$ and $\theta_1,\dots,\theta_\m$ are positive-degree homogeneous elements. 
Then, $\theta_1,\dots,\theta_\m$ is a \emph{regular sequence} if the equivalence class of $\theta_i$ is a non-zero-divisor in the quotient ring $R/(\theta_1,\dots,\theta_{i-1})$ for each $1 \leq i \leq \m$.  
\end{definition}

The following properties are well-known and useful in our setting. We refer the reader to \cite{Sta96}. (See also \cite[Proposition~5.1]{FHM} and \cite[Section~6]{AHHM}.)

\begin{lemma} \label{lemma:commutaive algebra}
Let $R=\bigoplus_{i=0}^{\infty} R_i$ be a graded $\Q$-algebra where each $R_i$ is a finite-dimensional $\Q$-vector space and $\theta_1,\dots,\theta_\m$ are positive-degree homogeneous elements in $R$. Then the followings hold:
\begin{enumerate}
\item[(i)] A sequence
$\theta_1,\dots,\theta_\m$ is a regular sequence in $R$ if and only if the Hilbert series $F(R/(\theta_1,\dots,\theta_\m),q)$ and $F(R,q)$ satisfy the following relation:
\begin{equation*} 
F(R/(\theta_1,\dots,\theta_\m),q)=F(R,q)\prod_{i=1}^\m(1-q^{\deg{\theta_i}}). 
\end{equation*}
\item[(ii)] A sequence $\theta_1,\dots,\theta_\m$ is a regular sequence in $R$ if and only if $\theta_1,\dots,\theta_\m$ is algebraically independent over $\Q$ and $R$ is a free $\Q[\theta_1,\dots,\theta_\m]$-module.
\item[(iii)] If $R$ is a polynomial ring $\Q[x_1, \dots , x_\rank]$, then $\theta_1,\dots,\theta_\rank$ is a regular sequence if and only if the solution set of the equations $\theta_1=0,\dots,\theta_\rank=0$ in $\C^\rank$ consists only of the
origin $\{0\}$. 
(Note that the number of elements $\theta_1,\dots,\theta_\rank$ is equal to the number of variables $x_1, \dots , x_\rank$.)
\end{enumerate}
\end{lemma}
 
\begin{lemma} \label{lemma:regular sequence tildeLIJ}
The set of sequences $\{\theta_k \mid k \in [\rank] \ \textrm{with} \ \alpha_k \in I \setminus J \} \cup \{\xi_j \mid  j \in [\rank] \ \textrm{with} \ \alpha_j \in J \} \cup \{\nu_i \mid i \in [\rank] \ \textrm{with} \ \alpha_i \in \Delta \setminus I \} \cup \{t \}$ defined in \eqref{eq:theta}, \eqref{eq:xi}, and \eqref{eq:nu} is a regular sequence in $\RR[t]$.
\end{lemma} 

\begin{proof}
Since $\RR[t]$ is the polynomial ring with $\rank+1$ variables $\varpi_1,\ldots,\varpi_\rank,t$ (or $\alpha_1,\ldots,\alpha_\rank,t$), by Lemma~\ref{lemma:commutaive algebra}-(iii) it is enough to prove that the solution set of the equations 
\begin{displaymath}
\left\{
\begin{array}{l}
\theta_k=0 \ \ \ \textrm{for} \ k \in [\rank] \ \textrm{with} \ \alpha_k \in I \setminus J \\
\xi_j=0 \ \ \ \textrm{for} \ j \in [\rank] \ \textrm{with} \ \alpha_j \in J \\
\nu_i=0 \ \ \ \textrm{for} \ i \in [\rank] \ \textrm{with} \ \alpha_i \in \Delta \setminus I \\
t=0
\end{array}
\right.
\end{displaymath}
in $\C^{\rank+1}$ consists only of the origin $\{0\}$. 
The equations above imply $\sum_{j \in [\rank]} \langle \alpha_i, \alpha_j \rangle \varpi_i \varpi_j=\alpha_i\varpi_i=0$ for all $i \in [\rank]$.  
As seen in the proof of \cite[Theorem~4.1]{HHM}, these equations have only the trivial solution as desired.
\end{proof} 

\begin{proposition} \label{proposition:Hilbert series LIJtilde}
The Hilbert series of $\RR[t]/\tilde \L_I^J$ is equal to
\begin{align*} 
F(\RR[t]/\tilde \L_I^J,\q)=\frac{(1+q^2)^{|I|-|J|}}{1-q^2}.
\end{align*}
\end{proposition} 

\begin{proof}
In general, if a sequence $\theta_1,\dots,\theta_\m$ is a regular sequence in a graded $\Q$-algebra $R$, then $\theta_1,\dots,\theta_{\m-1}$ is also a regular sequence in $R$ by Definition~\ref{definition:regular sequence}. 
In our setting, Lemma~\ref{lemma:regular sequence tildeLIJ} implies that the set of sequences $\{\theta_k \mid k \in [\rank] \ \textrm{with} \ \alpha_k \in I \setminus J \} \cup \{\xi_j \mid  j \in [\rank] \ \textrm{with} \ \alpha_j \in J \} \cup \{\nu_i \mid i \in [\rank] \ \textrm{with} \ \alpha_i \in \Delta \setminus I \}$ is a regular sequence in $\RR[t]$. 
It then follows from Lemma~\ref{lemma:commutaive algebra}-(i) that
\begin{align*} 
&F(\RR[t]/\tilde \L_I^J,\q)\\
=&F(\RR[t],\q) \left(\prod_{k \in [\rank] \atop \alpha_k \in I \setminus J} (1-\q^{\deg \theta_k}) \right) \left(\prod_{j \in [\rank] \atop \alpha_j \in J} (1-\q^{\deg \xi_k}) \right) \left(\prod_{i \in [\rank] \atop \alpha_i \in \Delta \setminus I} (1-\q^{\deg \nu_i}) \right) \\
=&\frac{1}{(1-\q^2)^{\rank+1}}(1-\q^4)^{|I|-|J|} (1-\q^2)^{|J|} (1-\q^2)^{\rank-|I|} \\
=&\frac{(1+\q^2)^{|I|-|J|}}{1-\q^2}.
\end{align*}
\end{proof} 

By \eqref{eq:HilbertSeriesHSYIJ} and Proposition~\ref{proposition:Hilbert series LIJtilde}, the both sides of \eqref{eq:psiIJ} have the same Hilbert series:
\begin{align} \label{eq:HilbertSeriesLIJtildeHSYIJ}
F(\RR[t]/\tilde \L_I^J, \q)=F(H^*_S(Y_I^J),\q)=\frac{(1+\q^2)^{|I|-|J|}}{1-\q^2}.
\end{align}

\begin{theorem} \label{theorem:psiIJiso}
The map $\psi_I^J$ in \eqref{eq:psiIJ} is an isomorphism.
\end{theorem}

\begin{proof}
This proof is based on the idea of \cite[Section~7]{AHHM}. 
Let $\tilde \L$ be the ideal $\tilde \L_{\Delta}^{\emptyset}$ in $\RR[t]$ and we denote by $\psi$ the map $\psi_{\Delta}^{\emptyset}$ defined in \eqref{eq:psiIJ} for $I = \Delta$ and $J=\emptyset$. 
Then, it follows from Theorem~\ref{theorem:EquivariantCohomologyPet} that the map $\psi$ is an isomorphism.
We also note that there is a natural surjective map $\RR[t]/\tilde \L \rightarrow \RR[t]/\tilde \L_I^J$ since $\tilde \L \subset \tilde \L_I^J$ by definition. 
Hence, we obtain the following commutative diagram:
\[
\begin{CD}
\RR[t]/\tilde \L @>\psi >\cong> H^*_S(\Pet_\Phi) @>>> H^*_S\big((\Pet_\Phi)^S\big) \\
@VV\text{surj}V @VVV @VV\text{surj}V \\
\RR[t]/\tilde \L_I^J @>\psi_I^J >> H^*_S(Y_I^J) @>>> H^*_S\big((Y_I^J)^S\big) 
\end{CD}
\]
where all arrows in the right commutative diagram denote the restriction maps.
Consider the localization of the rings in the commutative diagram above with respect to the multiplicatively closed set $\mathcal{T}\coloneqq \Q[t] \setminus \{0\}$ as follows: 
\[
\begin{CD}
\mathcal{T}^{-1}(\RR[t]/\tilde \L) @>\mathcal{T}^{-1}\psi >\cong> \mathcal{T}^{-1}H^*_S(\Pet_\Phi) @>>\cong> \mathcal{T}^{-1}H^*_S\big((\Pet_\Phi)^S\big) \\
@VV\text{surj}V @VVV @VV\text{surj}V \\
 \mathcal{T}^{-1}(\RR[t]/\tilde \L_I^J) @>\mathcal{T}^{-1}\psi_I^J >> \mathcal{T}^{-1}H^*_S(Y_I^J) @>>\cong> \mathcal{T}^{-1}H^*_S\big((Y_I^J)^S\big) 
\end{CD}
\]
Then, the horizontal maps in the right-hand square are isomorphisms by the localization theorem (\cite{Hsi}) which implies that the middle vertical map is surjective.
Thus, we see that $\mathcal{T}^{-1}\psi_I^J$ is also surjective. A comparison of Hilbert series shows that $\mathcal{T}^{-1}\psi_I^J$ is an isomorphism by \eqref{eq:HilbertSeriesLIJtildeHSYIJ}.  
Finally, consider the following commutative diagram:
\[
 \begin{CD}
 \RR[t]/\tilde \L_I^J @>\psi_I^J>> H^*_S(Y_I^J) \\
 @VV\text{inj}V @VV\text{inj}V \\
 \mathcal{T}^{-1}(\RR[t]/\tilde \L_I^J) @>\mathcal{T}^{-1}\psi_I^J >\cong > \mathcal{T}^{-1}H^*_S(Y_I^J)
 \end{CD}
 \]
Note that the rightmost vertical map is injective since $H^*_S(Y_I^J)$ is a free $\Q[t]$-module by \eqref{eq:free_HS(YIJ)}.
One can see that the leftmost vertical map is also injective. 
In fact, it follows from Lemma~\ref{lemma:regular sequence tildeLIJ} that $t$ is a non-zero-divisor in $\RR[t]/\tilde \L_I^J$ by the definition of a regular sequence. 
This implies that the left vertical arrow is injective.
Hence, we have the injectivity for $\psi_I^J$.
Therefore, we conclude from \eqref{eq:HilbertSeriesLIJtildeHSYIJ} that $\psi_I^J$ is an isomorphism. 
\end{proof}

\begin{proof}[Proof of Theorem~\ref{theorem:Cohomology1}]
Since $H^{odd}(Y_I^J)=0$ (Remark~\ref{remark:YIJ odd degree}), by setting $t=0$ we obtain the $\Q$-algebra isomorphism 
\begin{align*} 
H^*(Y_I^J) \cong \RR/\L_I^J,
\end{align*}
which sends $\alpha \in \RR/\L_I^J$ to $c_1(L_{\alpha}^*) \in H^*(Y_I^J)$.
In other words, the map $\varphi_I^J: \RR \xrightarrow{\varphi} H^*(\Pet_\Phi) \rightarrow H^*(Y_I^J)$ is surjective and its kernel coincides with the ideal $\L_I^J$.
\end{proof}

\bigskip 

\section{Intersections of Peterson variety and Schubert varieties} 
\label{section:Intersections of Peterson variety and Schubert varieties}

The intersections of Peterson variety and Schubert varieties for general Lie type have been much studied in \cite{GMS}. (See also \cite{AHKZ} for type $A$.)
In this section we first outline the work of \cite{GMS} and one easy way to attain the result of the cohomology ring of the intersections of Peterson variety and Schubert varieties. 
We also verify that the Schubert intersection satisfies the conditions (1) and (2) in Theorem~\ref{theorem:Cohomology1}. 

For $w \in W$, the Schubert cell $X_w^\circ$ is defined to be the $B$-orbit of $wB$ in $G/B$.
Recall that the Schubert variety $X_w$ is the closure of the Schubert cell $X_w^\circ$.
As is well-known, $X_w$ is decomposed into the disjoint union of Schubert cells $X_v^\circ$ ($v \leq w$ in Bruhat order).
The following result says which Schubert cells intersect with the Peterson variety.

\begin{lemma} $($\cite[Lemma~5.7]{HaTy17} and \cite[Corollary~4.13]{Pre13}$)$ \label{lemma:intersectionSchubertcellPet}
Let $v \in W$.
The Schubert cell $X_v^\circ$ does intersect with $\Pet_\Phi$ if and only if $v=w_K$ for some $K \subset \Delta$. 
\end{lemma}

For $I \subset \Delta=\{\alpha_1,\ldots,\alpha_\rank \}$, we define   
\begin{align} \label{eq:intersectionPetersonSchubert}
X_I^\circ:=X_{w_I}^\circ \cap \Pet_\Phi \ \textrm{and} \ X_I:=X_{w_I} \cap \Pet_\Phi.
\end{align}
The closure of $X_I^\circ$ is irreducible and its complex dimension is equal to $|I|$ (\cite[Proposition~3.2]{GMS}). 
As is well-known, the Schubert variety $X_{w_I}$ is the flag variety $G_I/B_I$ for the Levi subgroup $G_I$ generated by $I$ and $B_I=G_I \cap B$. 
The usual projection of the regular nilpotent element $N$ to the Levi subalgebra is the regular nilpotent element for the Levi subalgebra, and the intersection $X_I=X_{w_I} \cap \Pet_\Phi$ is then the Peterson variety for $G_I$. 
In particular, $X_I$ coincides with the closure of $X_I^\circ$ by the irreducibility of Peterson variety (\cite[Proposition~6.5]{GMS}). 
We record this property in the statement below.

\begin{proposition} $($\cite[Propositions~3.2 and 6.5]{GMS}$)$ \label{proposition:propertyXI}
For all $I \subset \Delta$, we have $X_I= \overline{X_I^\circ}$. Also, $X_I$ is the Peterson variety in the flag variety $G_I/B_I$, which is irreducible and $\dim_\C X_I = |I|$. 
\end{proposition}

\begin{remark}
The statement of Proposition~\ref{proposition:propertyXI} for type $A$ is proved by \cite[Proposition~3.4 and Corollary~3.6]{AHKZ}.
\end{remark}

As a consequence, one can calculate the cohomology ring for $X_I$ from Proposition~\ref{proposition:propertyXI} and Theorem~\ref{theorem:CohomologyPet} as follow.
 
\begin{corollary} \label{corollary:XI}
Let $I \subset \Delta$ and $X_I$ the intersection defined in \eqref{eq:intersectionPetersonSchubert}.
Then, the restriction map $H^*(\Pet_\Phi) \rightarrow H^*(X_I)$ is surjective. 
Moreover, there is an isomorphism of graded $\Q$-algebras
\begin{align*} 
H^*(X_I) \cong \RR/(\alpha_k \varpi_k \mid k \in [\rank] \ \textrm{with} \ \alpha_k \in I) +(\varpi_i \mid i \in [\rank] \ \textrm{with} \ \alpha_i \in \Delta \setminus I),
\end{align*}
which sends $\alpha$ to $c_1(L_\alpha^*)$. 
\end{corollary}

For the rest of this section, we confirm that $X_I$ satisfies the conditions~(1) and (2) in Theorem~\ref{theorem:Cohomology1} when $J$ is the empty set.

The following lemma is well-known fact (cf. \cite[Equation~(5.16)]{HaTy17}).

\begin{lemma} \label{lemma:Bruhat}
For two subsets $K$ and $I$ of $\Delta$, $w_K \leq w_I$ if and only if $K \subset I$. 
\end{lemma}

\begin{proof}
If $K \subset I$, then $w_K$ belongs to the parabolic subgroup $W_I$. Since $w_I$ is the longest element of $W_I$, one has $w_K \leq w_I$.
Conversely, if $w_K \leq w_I$, then we show that $K \subset I$.
Take a simple root $\alpha_k \in K$.
Since $w_K$ is the longest element of $W_K$, the simple reflection $s_k$ satisfies $s_k \leq w_K$. This and the assumption $w_K \leq w_I$ yields that $s_k \leq w_I$. 
Thus, the simple root $\alpha_k$ must be an element of $I$. 
\end{proof}

It follows from Lemma~\ref{lemma:intersectionSchubertcellPet} that $X_I$ can be decomposed as follows:
\begin{align} \label{eq:decomp XI}
X_I=\left( \coprod_{v \leq w_I} X_v^\circ \right) \cap \Pet_\Phi= \coprod_{w_K \leq w_I} \left( X_{w_K}^\circ \cap \Pet_\Phi \right) =\displaystyle\coprod_{K \subset I} X_{K}^\circ.
\end{align}
Here, for the last equality we have used Lemma~\ref{lemma:Bruhat}.

\begin{lemma} \label{lemma:condition(1)XI}
The $S$-action on $\Pet_{\Phi}$ preserves $X_I$ and the set of $S$-fixed points is 
\begin{align*}
(X_I)^S =\{w_K \mid K \subset I \}.
\end{align*}
\end{lemma}

\begin{proof}
Since the natural action of the maximal torus $T$ on $G/B$ preserves $X_{w_I}$, the Schubert variety $X_{w_I}$ admits the $S$-action by restriction.
Thus, the $S$-action on $\Pet_\Phi$ preserves $X_I$ by the definition \eqref{eq:intersectionPetersonSchubert}. 
Noting that the $S$-fixed point set $(X_I)^S$ is equal to the intersection $(\Pet_\Phi)^S \cap X_{w_I}$, $(X_I)^S$ coincides with the fixed points $w_K$ with $w_K \leq w_I$ in Bruhat order by Lemma~\ref{lemma:PetFixedPoints}.
Hence, it follows from Lemma~\ref{lemma:Bruhat} that $(X_I)^S =\{w_K \mid K \subset I \}$.
\end{proof}

Next we compute the Poincar\'e polynomial of $X_I$. 
The following proposition is useful for the computation of the Poincar\'e polynomial.

\begin{proposition} $($\cite[Section~B.3]{Ful97}, \cite[Examples~1.9.1 and 19.1.11]{Ful98}$)$ \label{proposition:paving}
If an algebraic variety $Z$ has a filtration $Z=Z_m \supset Z_{m-1} \supset \dots \supset Z_0 \supset Z_{-1}=\emptyset$ by closed subsets, with each $Z_i \setminus Z_{i-1}$ a disjoint union of $U_{i,j}$ isomorphic to an affine space $\C^{n(i,j)}$, then the homology classes $[\overline{U_{i,j}}]$ of the closures of $U_{i,j}$ forms a basis of $H_*(Z)$. 
\end{proposition}

\begin{proposition} \label{proposition:condition(2)XI}
The Poincar\'e polynomial of $X_I$ is equal to 
\begin{align*}
\Poin(X_I,\q)=(1+\q^2)^{|I|}.
\end{align*}
\end{proposition}

\begin{proof}
Let $Z_k=\displaystyle\bigcup_{K \subset I \atop |K| \leq k} X_K$ for each $1 \leq k \leq |I|$. 
Consider the filtration $X_I=Z_{|I|} \supset Z_{|I|-1} \supset \dots \supset Z_0 \supset Z_{-1}=\emptyset$.
Then, the decomposition \eqref{eq:decomp XI} yields that each $Z_k \setminus Z_{k-1}$ is a disjoint union of $X_{K}^\circ$ with $|K|=k$ and $K \subset I$. 
Since it follows from \cite[Theorem~4.10 and Corollary~4.13]{Pre13} that $X_{K}^\circ \cong \C^{|K|}$ for all $K \subset \Delta$, we obtain $H_{odd}(X_I)=0$ and $\dim_\Q H_{2k}(X_I)=\tbinom{|I|}{k}$ by Proposition~\ref{proposition:paving}.
Therefore, the result follows from the universal coefficient theorem. 
\end{proof}

Lemma~\ref{lemma:condition(1)XI} and Proposition~\ref{proposition:condition(2)XI} are nothing but the conditions~(1) and (2) in Theorem~\ref{theorem:Cohomology1} when $J$ is the empty set. 
Therefore, we also obtain Corollary~\ref{corollary:XI} by using Theorem~\ref{theorem:Cohomology1}.

\begin{remark}
We outline an alternative approach by \cite{RWY} to attain Corollary~\ref{corollary:XI}. 
In \cite{RWY} we understand the Kronecker dual to the homology basis given by the Schubert varieties for presenting the cohomology rings of Schubert varieties.
In the case of the Peterson variety, the homology classes $\{[X_I]\}_{I \subset \Delta}$ forms a basis for $H_*(\Pet_\Phi)$ as discussed above.
In more general, $\{[X_K]\}_{K \subset I}$ forms a basis for $H_*(X_I)$ for any $I \subset \Delta$. 
We denote by $p_I$ the image of the Schubert class $\sigma_{v_I}$ associated with Coxeter element $v_I$ under the restriction map $H^*(G/B) \to H^*(\Pet_\Phi)$.
By \cite{Dre15} the set $\{p_I \}_{I \subset \Delta}$ is a basis of $H^*(\Pet_\Phi)$. 
Note that this fact for type $A$ is proved by \cite{HaTy11}.
We also note that \cite{AHKZ} gives a geometric interpretation for the basis in type $A$. 
In fact, $p_I$ reflects the geometry of $\Omega_I$ in type $A$. 
Goldin--Mihalcea--Singh proves in \cite{GMS} that $\{\frac{p_I}{m(v_I)} \}_{I \subset \Delta}$ is the Kronecker dual to $\{[X_I]\}_{I \subset \Delta}$ for any Lie types where $m(v_I)$ is the multiplicity of the (unique) intersection point of $\Omega_{v_I} \cap X_I$. 
Consequently, the restriction map $H^*(\Pet_\Phi) \to H^*(X_I)$ is surjective with kernel 
\begin{align*}
\mathcal{K}_I \coloneqq \Q \textrm{-span of} \ \{p_K \mid K \not\subset I \}. 
\end{align*}
This yields the isomorphism $H^*(X_I) \cong H^*(\Pet_\Phi)/\mathcal{K}_I$.
What we want to show is that the ideal $\mathcal{K}_I$ is generated by $c_1(L_{\varpi_i}^*)$'s for $i \in [\rank]$ with $\alpha_i \in \Delta \setminus I$.
One can prove this by using Giambelli's formula and Monk's formula in Peterson variety given by \cite{Dre15}.
In fact, since $p_K$ equals $\prod_{k \in K} c_1(L_{\varpi_k}^*)$ up to a scalar multiplication by the Giambelli's formula, $\mathcal{K}_I$ is included in the ideal $(c_1(L_{\varpi_i}^*) \mid i \in [\rank] \ \textrm{with} \ \alpha_i \in \Delta \setminus I)$. 
In order to prove the opposite inclusion $(c_1(L_{\varpi_i}^*) \mid i \in [\rank] \ \textrm{with} \ \alpha_i \in \Delta \setminus I) \subset \mathcal{K}_I$, it suffices to show that $p_L \cdot c_1(L_{\varpi_i}^*)$ belongs to $\mathcal{K}_I$ for all $L \subset \Delta$ and $i \in [\rank]$ with $\alpha_i \in \Delta \setminus I$ since $\{p_L\}_{L \subset \Delta}$ forms a basis of $H^*(\Pet_\Phi)$. 
However, this follows from the Monk's formula, so we conclude that 
\begin{align*}
H^*(X_I) &\cong H^*(\Pet_\Phi)/\mathcal{K}_I \\
&\cong \RR/(\alpha_k \varpi_k \mid k \in [\rank] \ \textrm{with} \ \alpha_k \in I) +(\varpi_i \mid i \in [\rank] \ \textrm{with} \ \alpha_i \in \Delta \setminus I).
\end{align*}
Note that we used Theorem~\ref{theorem:CohomologyPet} for the second isomorphism above. 
\end{remark}

\bigskip

\section{Intersections of Peterson variety and opposite Schubert varieties in type $A$} \label{section:Intersections of Peterson variety and opposite Schubert varieties}

In this section we consider the intersections of Peterson variety and opposite Schubert varieties in type $A_{n-1}$, which is the opposite situation to Section~\ref{section:Intersections of Peterson variety and Schubert varieties}. 
We aim to describe a presentation of the cohomology rings of the special cases of their intersections.

The flag variety $\Fl(\C^n)$ in type $A_{n-1}$ is the set of nested complex vector spaces $V_\bullet \coloneqq (V_1 \subset V_2 \subset \cdots \subset V_n = \C^n)$ where each $V_i$ is an $i$-dimensional subspace of $\C^n$.
Let $N$ be a regular nilpotent matrix, which is a matrix whose Jordan form consists of exactly one Jordan block with corresponding eigenvalue equal to $0$.
The Peterson variety $\Pet_n$ in type $A_{n-1}$ is defined by 
\begin{align} \label{eq:Pet_typeA}
\Pet_n \coloneqq \{V_\bullet \in \Fl(\C^n) \mid NV_i \subset V_{i+1} \ \textrm{for all} \ i=1,\ldots,n-1 \}.
\end{align}
We may assume that $N$ is in Jordan form. 
As is well-known, $\Fl(\C^n)$ can be identified with $SL_n(\C)/B$ where $B$ is the set of the upper triangular matrices in the special linear group $SL_n(\C)$. 
Under the identification $SL_n(\C)/B \cong \Fl(\C^n)$, we can write 
\begin{align*}
\Pet_n = \{gB \in SL_n(\C)/B \mid  (g^{-1}Ng)_{i,j} = 0 \ \textrm{for} \ i, j \in [n] \ \textrm{with} \ i > j+1 \},
\end{align*}
where $(g^{-1}Ng)_{i,j}$ denotes the $(i,j)$-th component of the matrix $g^{-1}Ng$. 

Let $T$ be the set of the diagonal matrices in $SL_n(\C)$ and $\perm_n$ the permutation group on $[n]=\{1,\ldots,n\}$. 
The maximal torus $T$ naturally acts on the flag variety $\Fl(\C^n) \cong SL_n(\C)/B$ and the fixed point set $\big(\Fl(\C^n)\big)^T$ consists of the permutation flags $V^{(w)}_\bullet=(V^{(w)}_i)_{i\in[n]}$ given by $V^{(w)}_i=\textrm{span}(e_{w(1)},e_{w(2)},\ldots,e_{w(i)})$ where $e_1, e_2, \ldots, e_n$ denotes the standard basis of $\C^n$.
We may regard a permutation $w$ as the permutation matrix (up to sign) in $SL_n(\C)$.
We define $S$ by the following one-dimensional subtorus of $T$:  
\begin{align*}
S &= 
\left\{
\left.
\begin{pmatrix}
g^{-m} & & & & & & \\
 & \ddots & & & & & \\
 & & g^{-1} & & & & \\
 & & & 1 & & & \\
 & & & & g & & \\
 & & & & & \ddots & \\
 & & & & & & g^m \\
\end{pmatrix}
\in T\ 
\right| \ g\in \C^*
\right\} \ \ \ \textrm{if} \ n=2m+1 \ \textrm{odd}, \\
S &= 
\left\{
\left.
\begin{pmatrix}
g^{-(2m-1)} & & & & & & & \\
 & \ddots & & & & & & \\
 & & g^{-3} & & & & & \\
 & & & g^{-1} & & & & \\
 & & & & g & & & \\
 & & & & & g^3 & & \\
 & & & & & & \ddots & \\
 & & & & & & & g^{2m-1}\\
\end{pmatrix}
\in T\ 
\right| \ g\in \C^*
\right\} \ \ \ \textrm{if} \ n=2m \ \textrm{even}. 
\end{align*}
Then, one can easily check that the one-dimensional torus $S$ above preserves $\Pet_n$.

\begin{remark}
If $G=GL_n(\C)$, then the one-dimensional torus $S$ in $G$ which preserves $\Pet_n$ is often taken as $S=\{\diag(g,g^2,\ldots,g^n) \mid g \in \C^* \}$ (e.g. \cite{AHHM, AHKZ}).
However, we now consider the case when $G=SL_n(\C)$, so the definition of $S$ is modified as above.
\end{remark}

For $i \in [n-1]$, let $\alpha_i$ be the character of $T$ defined by $\diag(g_1,\ldots,g_n) \mapsto g_i g_{i+1}^{-1}$. 
Recall that we denote by $\Delta$ the set of simple roots, i.e. $\Delta=\{\alpha_1, \ldots, \alpha_{n-1} \}$ in this case.
Here and below, we identify the vertices $\Delta=\{\alpha_1, \ldots, \alpha_{n-1} \}$ of the Dynkin diagram of type $A_{n-1}$ with $[n-1]$ for simplicity.
For this reason, we use symbols $I,J,K$ as subsets of $[n-1]$. (Note that the symbols $I,J,K$ are used for subsets of $\Delta$ in previous sections.)
Every subset $K \subset [n-1]$ can be decompose into the connected components of the Dynkin diagram of type $A_{n-1}$:
\begin{align} \label{eq:connected components}
K = K_1 \sqcup K_2 \sqcup \cdots \sqcup K_m.
\end{align}
In other words, each $K_i \ (1 \leq i \leq m)$ means a maximal consecutive subset of $[n-1]$. 
For each connected component $K_i$, one can define the permutation subgroup $\perm_{K_i}$ on $K_i$ and the longest element $w_0^{(K_i)}$ of $\perm_{K_i}$.
Then, we define the permutation $w_K \in \perm_n$ by  
\begin{align*}
w_K \coloneqq w_0^{(K_1)} w_0^{(K_2)} \cdots w_0^{(K_m)}. 
\end{align*}

\begin{example}\label{eg: wJ}
{\rm
Let $n=9$ and $K=\{1,2\} \sqcup\{5,6,7\}$. 
Then, the one-line notation of $w_K$ is given by 
\begin{align*}
w_K=321487659.
\end{align*}
By identifying the permutation $w_K$ with its permutation matrix, we have 
\begin{align*}
 w_K 
 =
 \left(
 \begin{array}{@{\,}ccc|c|cccc|cc@{\,}}
     & & 1 & & & & & & \\
     & 1 & & & & & & & \\ 
    1 & & & & & & & & \\ \hline
     & & & 1 & & & & &  \\ \hline
     & & & & & & & 1&  \\
     & & & & &  & 1& & \\
     & & & & & 1& & & \\
     & & & & 1& & & & \\ \hline
     & & & & & & & & 1 
 \end{array}
 \right).
\end{align*}}
\end{example}

It follows from Lemma~\ref{lemma:PetFixedPoints} that the $S$-fixed point set $(\Pet_n)^S$ is given by 
\begin{align*}
(\Pet_n)^S=\{w_K \in \perm_n \mid K \subset [n-1] \}.
\end{align*}
Recall that a Schubert cell $X_v^\circ$ intersects with the Peterson variety $\Pet_n$ if and only if $v$ is of the form $w_K$ for some $K \subset [n-1]$ (Lemma~\ref{lemma:intersectionSchubertcellPet}). 
In type $A$, the intersections $X_{w_K}^\circ \cap \Pet_n$ have a good description (\cite{Ins, IY}).
We first explain the description of $X_{w_K}^\circ \cap \Pet_n$ for the case when $w_K$ is the longest element $w_0$ (i.e. $K=[n-1]$). 
To do that, we recall that each flag $V_\bullet$ in the Schubert cell $X_{w_0}^\circ$ has $V_j$ spanned by the first $j$ columns of a matrix with $1$'s in the $(w_0(j),j)$ positions, and $0$'s to the right of  these $1$'s, which is described as follows (cf. \cite[Section~10.2]{Ful97}): 
\begin{align} \label{eq:Xw0cell}
\left(
 \begin{array}{@{\,}ccccccc@{\,}}
     x_{11} & x_{12} & x_{13} & \cdots & x_{1 \, n-2} & x_{1 \, n-1} & 1\\
     x_{21} & x_{22} & \cdots & x_{2 \, n-3} & x_{2 \, n-2} & 1& \\ 
     x_{31} & \vdots & \iddots & x_{3 \, n-3} & 1& & \\ 
     \vdots& \iddots & \iddots & \iddots & & &  \\
     x_{n-2 \, 1} & x_{n-2 \, 2} & 1& & & & \\
     x_{n-1 \, 1} & 1& & & & & \\ 
     1& & & & & &  
 \end{array}
 \right). 
\end{align}
Here, $x_{ij}$'s above denote arbitrary complex numbers.
We may regard the Schubert cell $X_{w_0}^\circ$ as the set of matrices of \eqref{eq:Xw0cell}. 
One can verify from the definition \eqref{eq:Pet_typeA} that equations obtained by intersecting with $\Pet_n$ are given by 
\begin{align*}
x_{1k}=x_{2 \, k-1}=x_{3 \, k-2}=\cdots=x_{k1} 
\end{align*}
for all $1 \leq k \leq n-1$.
Namely, each flag $V_\bullet$ in the intersection $X_{w_0}^\circ \cap \Pet_n$ has $V_j$ spanned by the first $j$ columns of the following matrix
\begin{align} \label{eq:Xw0capPet}
\left(
 \begin{array}{@{\,}ccccccc@{\,}}
     x_1 & x_2 & x_3 & \cdots & x_{n-2} & x_{n-1} & 1\\
     x_2 & x_3 & \cdots & x_{n-2} & x_{n-1} & 1& \\ 
     x_3 & \vdots & \iddots & x_{n-1} & 1& & \\ 
     \vdots& \iddots & \iddots & \iddots & & &  \\
     x_{n-2} & x_{n-1} & 1& & & & \\
     x_{n-1} & 1& & & & & \\ 
     1& & & & & &  
 \end{array}
 \right), 
\end{align}
where $x_1,\ldots,x_{n-1}$ are arbitrary complex numbers. 
Note that $X_{w_0}^\circ \cap \Pet_n$ can be regarded as the set of matrices \eqref{eq:Xw0capPet}. 
This yields an isomorphism $X_{w_0}^\circ \cap \Pet_n \cong \C^{n-1}$. 
In general, for a subset $K$ of $[n-1]$, the Schubert cell $X_{w_K}^\circ$ consists of flags $V_\bullet=(V_j)_{j \in [n]}$ where $V_j$ is spanned by the first $j$ columns of a matrix with $1$'s in the $(w_K(j),j)$ positions, and $0$'s to the right of  these $1$'s, but with $0$'s under these $1$'s in this case (cf. \cite[Section~10.2]{Ful97}). 
In other words, the matrix which represents a flag in $X_{w_K}^\circ$ forms a block diagonal matrix such that each block is in the form of \eqref{eq:Xw0cell} of smaller size. 
We can think of $X_{w_K}^\circ$ as the set of such block diagonal matrices.

\begin{example}
If $n=9$ and $K=\{1,2\} \sqcup\{5,6,7\}$, then $X_{w_K}^\circ$ can be regarded as the set of the following matrices
\begin{align*}
 \left(
 \begin{array}{@{\,}ccc|c|cccc|cc@{\,}}
     x_{11} & x_{12} & 1 & & & & & & \\
     x_{21} & 1 & & & & & & & \\ 
    1 & & & & & & & & \\ \hline
     & & & 1 & & & & &  \\ \hline
     & & & & y_{11} & y_{12} & y_{13} & 1&  \\
     & & & & y_{21} & y_{22} & 1& & \\
     & & & & y_{31} & 1& & & \\
     & & & & 1& & & & \\ \hline
     & & & & & & & & 1 
 \end{array}
 \right), 
\end{align*}
where $x_{ij}$'s and $y_{ij}$'s above are arbitrary complex numbers.
\end{example}

Consider the decomposition \eqref{eq:connected components} into connected components. 
Let $C^{(K_i)}$ be the set of matrices of size $|K_i|+1$ whose nonzero variables repeat along the $|K_i|$-antidiagonals lying to the left of the main antidiagonal. (Note that if $K=[n-1]$, then $C^{(K)}$ for the unique connected component $K$ consists of the matrices in \eqref{eq:Xw0capPet}.)
It then follows from \eqref{eq:Pet_typeA} that $X_K^\circ$ consists of flags $V_\bullet=(V_j)_{j \in [n]}$ so that $V_j$ is spanned by the first $j$ columns of the following block diagonal matrix: 
\begin{align} \label{eq:XwKcapPet}
\left(
 \begin{array}{@{\,}c|c|c|ccc|c|c|c|ccc|c@{\,}}
     1 & & & & & & & & & & & & \\ \hline
     & \ddots & & & &  &  & & & & & & \\ \hline
     & & 1 & & &  & & & & & & & \\ \hline
     & & & & & & & & & & & & \\
     & & & & g_1 & & & & & & & & \\ 
     & & & & & & & & & & & & \\ \hline
     & & & & & & 1 & & & & & & \\ \hline
     & & & & & & & \ddots &  &  & & & \\ \hline
     & & & & & & & & 1 & & & & \\ \hline
     & & & & & & & & & & & & \\
     & & & & & & & & & & g_2 & & \\ 
     & & & & & & & & & & & & \\ \hline
     & & & & & & & & & & & & \ddots 
 \end{array}
 \right),
\end{align}
where $g_i \in C^{(K_i)}$ and the row numbers (or equivalently the column numbers) of $g_i$ belongs to $K_i \cup \{\max K_i +1 \}$ for all $1 \leq i \leq m$.
(See Example~\ref{eg: PetersonSchubertcell} below.) 
We may regard $X_K^\circ$ as the set of the block diagonal matrices in \eqref{eq:XwKcapPet}.
Note that one can directly see the following isomorphism
\begin{align*} 
X_K^\circ=X_{w_K}^\circ \cap \Pet_n \cong C^{(K_1)} \oplus C^{(K_2)} \oplus \cdots \oplus C^{(K_m)} \cong \C^{|K|}, 
\end{align*}
which is the special case of a result in \cite{Tym06}. 

\begin{example}\label{eg: PetersonSchubertcell}
Continuing Example~\ref{eg: wJ}, each flag in $X_K^\circ$ for $K=\{1,2\} \sqcup\{5,6,7\}$ and $n=9$ is represented by the following matrix
\begin{align*}
 \left(
 \begin{array}{@{\,}ccc|c|cccc|cc@{\,}}
     x_1 & x_2 & 1 & & & & & & \\
     x_2 & 1 & & & & & & & \\ 
    1 & & & & & & & & \\ \hline
     & & & 1 & & & & &  \\ \hline
     & & & & y_1 & y_2 & y_3 & 1&  \\
     & & & & y_2 & y_3 & 1& & \\
     & & & & y_3 & 1& & & \\
     & & & & 1& & & & \\ \hline
     & & & & & & & & 1 
 \end{array}
 \right),
\end{align*}
where $ x_1,x_2,y_1,y_2,y_3 \in \C$. 
If we write $K_1=\{1,2\}$ and $K_2=\{5,6,7\}$, then 
\begin{align*}
C^{(K_1)}= \left\{
\left.
\begin{pmatrix}
     x_1 & x_2 & 1 \\
     x_2 & 1 & \\ 
    1 & &  
\end{pmatrix}
 \ \right| \ x_1, x_2 \in \C \right\} \ \textrm{and} \ 
C^{(K_2)}= \left\{
\left.
\begin{pmatrix}
    y_1 & y_2 & y_3 & 1  \\
    y_2 & y_3 & 1&  \\
    y_3 & 1& &  \\
    1& & &  
\end{pmatrix}
 \ \right| \ y_1, y_2, y_3 \in \C \right\}.
\end{align*}
One can verify that $X_K^\circ \cong C^{(K_1)} \oplus C^{(K_2)} \cong \C^5$.  
\end{example}

We now consider the intersections of Peterson variety and \textit{oppossite} Schubert varieties $\Omega_w$.
Let $B^{-}$ denote the set of the lower triangular matrices in $SL_n(\C)$. 
For a permutation $w \in \perm_n$, the opposite Schubert cell $\Omega_w^\circ$ is defined as the $B^{-}$-orbit of the permutation flag $wB$ in $\Fl(\C^n)$.
The opposite Schubert variety $\Omega_w$ is the closure of $\Omega_w^\circ$.
Since $B^{-}=w_0Bw_0$, we have the decomposition $\Omega_w = \sqcup_{v \geq w} \, \Omega_v^\circ$ ($v \geq w$ in Bruhat order).
The following lemma is an analogue of Lemma~\ref{lemma:intersectionSchubertcellPet}.

\begin{lemma} $($\cite[Lemma~3.5]{AHKZ}$)$ \label{lemma:intersection_oppositeSchubertcellPet}
Let $v \in \perm_n$.
The opposite Schubert cell $\Omega_v^\circ$ does intersect with $\Pet_n$ if and only if $v=w_K$ for some $K \subset [n-1]$. 
\end{lemma}

For $J \subset [n-1]$, we define   
\begin{align} \label{eq:intersectionPeterson_oppositeSchubert}
\Omega_J^\circ:=\Omega_{w_J}^\circ \cap \Pet_n \ \textrm{and} \ \Omega_J:=\Omega_{w_J} \cap \Pet_n.
\end{align}

\begin{proposition} $($\cite[Corollary~3.14]{AHKZ}$)$ \label{proposition:propertyOmegaJ}
For any subset $J$ of $[n-1]$, $\Omega_J$ is equidimensional and its codimension in $\Pet_n$ is equal to $|J|$. 
\end{proposition}

\begin{remark}
The proof of \cite[Corollary~3.14]{AHKZ} implies that $\Omega_J$ is equidimensional.
\end{remark}

By using Lemmas~\ref{lemma:Bruhat} and \ref{lemma:intersection_oppositeSchubertcellPet}, one can see the following decomposition
\begin{align} \label{eq:decomp OmegaJ}
\Omega_J=\left( \coprod_{v \geq w_J} \Omega_v^\circ \right) \cap \Pet_n= \coprod_{w_K \geq w_J} \left( \Omega_{w_K}^\circ \cap \Pet_n \right) =\displaystyle\coprod_{K \supset J} \Omega_{K}^\circ.
\end{align}

By a similar argument of Lemma~\ref{lemma:condition(1)XI}, we can prove the following lemma.

\begin{lemma} \label{lemma:condition(1)OmegaJ}
Let $J$ be a subset of $[n-1]$.
The $S$-action on $\Pet_n$ preserves $\Omega_J$ and the fixed point set is given by
\begin{align*}
(\Omega_J)^S =\{w_K \mid K \supset J \}.
\end{align*}
\end{lemma}

The following decomposition is useful for computing the Poincar\'e polynomial of some $\Omega_J$.

\begin{lemma} \label{lemma:decomp_OmegaJ}
Let $J$ be a subset of $[n-1]$.
Then $\Omega_J$ has the following decomposition
\begin{align*}
\Omega_J=\displaystyle\coprod_{K \supset J} (\Omega_J \cap X_K^\circ).
\end{align*}
\end{lemma}

\begin{proof}
It follows from \cite[Proposition~3.2]{AHKZ} that 
\begin{align} \label{eq:intersectionXIOmegaJ}
X_I \cap \Omega_J \neq \emptyset \ \textrm{if and only if} \ I \supset J.
\end{align}
The property \eqref{eq:intersectionXIOmegaJ} yields that $X_I^\circ \cap \Omega_J = \emptyset$ if $I \not\supset J$.
Also, if $I\supset J$, then $w_I \in X_I^\circ \cap \Omega_J$ by Lemma~\ref{lemma:condition(1)OmegaJ}, in particular we have $X_I^\circ \cap \Omega_J \neq \emptyset$. 
Hence, we obtain 
\begin{align*} 
\Omega_J &= \Omega_J \cap \Pet_n = \Omega_J \cap \big(\displaystyle\coprod_{K \subset [n-1]} X_K^\circ \big) \ \ \ (\textrm{by} \  \eqref{eq:decomp XI}) \\
&=\displaystyle\coprod_{K \supset J} (\Omega_J \cap X_K^\circ). 
\end{align*}
\end{proof}

In what follows, we concentrate on the case when $J$ is of the form $[\bb]=\{1,2,\ldots,\bb \}$ for some $\bb$. 

\begin{lemma} \label{lemma:Omegab}
Let $\bb$ be a positive integer with $1 \leq \bb \leq n-1$. 
For $K \supset [\bb]$, we have
\begin{align*}
\Omega_{[\bb]} \cap X_K^\circ \cong \C^{|K|-\bb}.
\end{align*}
\end{lemma}

\begin{proof}
Recall that $X_K^\circ$ can be regarded as the set of the block diagonal matrices in \eqref{eq:XwKcapPet} so that $X_K^\circ \cong \C^{|K|}$. 
We now see equations that define $\Omega_{[\bb]}$ in $X_K^\circ \cong \C^{|K|}$.
By \cite[Proposition~3.10]{AHKZ} we have $\Omega_J=\cap_{j \in J} \Omega_{\{j\}}$ for arbitrary $J \subset [n-1]$. 
Hence, one can write
\begin{align} \label{eq:Omegab}
\Omega_{[\bb]} = \Omega_{\{1\}} \cap \Omega_{\{2\}} \cap \cdots \cap \Omega_{\{\bb\}}
\end{align}
in this case.
By \cite[Proposition~3.12]{AHKZ} the defining equation of $\Omega_{\{j\}}$ in $X_K^\circ$ is given by
\begin{align} \label{eq:defining equation Omegaj} 
\textrm{the leading principal minor of order} \ j \ \textrm{of} \ g = 0 \ \ \ (g \in X_K^\circ).
\end{align}
By the assumption $K \supset [\bb]$, the first block $g_1 \in C^{(K_1)}$ is located in the top-left corner of $g \in X_K^\circ$ as follows:
\begin{align} \label{eq:blockg1}
g= \left(
 \begin{array}{@{\,}ccc|c|c|c|ccc|c@{\,}}
      & & & & & & & & & \\
      & g_1 & & & & & & & & \\ 
     & & & & & & & & & \\ \hline
     & & & 1 & & & & & & \\ \hline
     & & & & \ddots &  &  & & & \\ \hline
     & & & & & 1 & & & & \\ \hline
     & & & & & & & & & \\
     & & & & & & & g_2 & & \\ 
      & & & & & & & & & \\ \hline
     & & & & & & & & & \ddots 
 \end{array}
 \right).
\end{align}
Note that $|K_1| \geq \bb$ since $K_1 \supset [\bb]$.
We write $g_1 \in C^{(K_1)}$ as
\begin{align*}
g_1=\left(
 \begin{array}{@{\,}ccccccc@{\,}}
     x_1 & x_2 & \cdots & x_\bb & \cdots & x_{|K_1|} & 1\\
     x_2 & & \iddots &  & \iddots & 1 & \\ 
     \vdots& \iddots & & \iddots & \iddots & &  \\
     x_\bb &  & \iddots & \iddots & & & \\ 
     \vdots & \iddots & \iddots & & & & \\
     x_{|K_1|} & 1 & & & & & \\ 
     1& & & & & & 
 \end{array}
 \right). 
\end{align*}
By \eqref{eq:Omegab}, \eqref{eq:defining equation Omegaj}, and \eqref{eq:blockg1}, the equations that define $\Omega_{[\bb]}$ in $X_K^\circ \cong \C^{|K|}$ are given by equations that the leading principal minor of order $j$ of the matrix $g_1$ is zero for all $1 \leq j \leq \bb$.
This implies that $x_1=x_2=\cdots=x_\bb=0$. 
Indeed, the case when $j=1$ yields $x_1=0$. 
Then the leading principal minor of order $2$ of the matrix $g_1$ is $x_1x_3-x_2^2=-x_2^2$, so the case when $j=2$ implies $x_2=0$.
Proceeding with this argument, we can inductively obtain that $x_1=x_2=\cdots=x_\bb=0$.  
Therefore, we conclude that $\Omega_{[\bb]} \cap X_K^\circ \cong \C^{|K|-\bb}$, as desired.
\end{proof}

\begin{proposition} \label{proposition:condition(2)OmegaJ}
For a positive integer $1 \leq \bb \leq n-1$, the Poincar\'e polynomial of $\Omega_{[\bb]}$ is given by
\begin{align*}
\Poin(\Omega_{[\bb]},\q)=(1+\q^2)^{n-1-b}.
\end{align*}
\end{proposition}

\begin{proof}
The proof is similar to that of Proposition~\ref{proposition:condition(2)XI}. 
Set $Z_k=\displaystyle\bigcup_{K \supset [\bb] \atop |K| \leq k} (\Omega_{[\bb]} \cap X_K)$ for each $\bb \leq k \leq n-1$. 
By \eqref{eq:decomp XI} and Lemma~\ref{lemma:decomp_OmegaJ} we see that $\Omega_{[\bb]}=Z_{n-1}$.
For the filtration $\Omega_{[\bb]}=Z_{n-1} \supset Z_{n-2} \supset \dots \supset Z_{\bb} \supset Z_{\bb-1}=\emptyset$, each $Z_k \setminus Z_{k-1}$ is a disjoint union of $\Omega_{[\bb]} \cap X_{K}^\circ$ with $|K|=k$ and $K \supset [\bb]$ by \eqref{eq:decomp XI}. 
It follows from Proposition~\ref{proposition:paving} and Lemma~\ref{lemma:Omegab} that 
\begin{align*}
\Poin(\Omega_{[\bb]},\q)=(1+\q^2)^{n-1-b}.
\end{align*}
\end{proof}

We now give an explicit presentation of the cohomology ring of $\Omega_{[\bb]}$.

\begin{theorem} \label{theorem:mainOmegaJ}
Let $\bb$ be a positive integer with $1 \leq \bb \leq n-1$ and consider the intersection $\Omega_{[\bb]}$ defined in \eqref{eq:intersectionPeterson_oppositeSchubert}. 
Then, the restriction map $H^*(\Pet_n) \rightarrow H^*(\Omega_{[\bb]})$ is surjective. 
Furthermore, there is an isomorphism of graded $\Q$-algebras
\begin{align*} 
H^*(\Omega_{[\bb]}) \cong \RR/(\alpha_k \varpi_k \mid \bb+1 \leq k \leq n-1) +(\alpha_j \mid 1 \leq j \leq \bb),
\end{align*}
which sends $\alpha$ to $c_1(L_\alpha^*)$. 
\end{theorem}

\begin{proof}
Lemma~\ref{lemma:condition(1)OmegaJ} and Proposition~\ref{proposition:condition(2)OmegaJ} are exactly the conditions~(1) and (2) in Theorem~\ref{theorem:Cohomology1} when $I=[n-1]$ and $J=[\bb]$. 
Therefore, the statement follows from Theorem~\ref{theorem:Cohomology1}.
\end{proof}

\begin{remark}
Theorem~\ref{theorem:psiIJiso} leads us to an explicit presentation of the $S$-equivariant cohomology ring $H^*_S(\Omega_{[\bb]})$.
Here, $t \in H^2_S(\pt)$ is defined as follows. 
If $n=2m+1$ is odd, then we define a character $\rho \in \Hom(S,\C^*)$ by $\rho: \diag(g^{-m},...,g^{-1},1,g,...,g^m) \mapsto g^{-1}$ and $t=c_1^S(L_\rho)$.
If $n=2m$ is even, then we define $\rho: \diag(g^{-(2m-1)},...,g^{-3},g^{-1},g,g^3...,g^{2m-1}) \mapsto g^{-2}$ and $t=c_1^S(L_\rho)$. 
Then, one can verify that the restriction map $H^*_T(\pt) \mapsto H^*_S(\pt)$ sends $\alpha_i^T$ to $t$ for all $i \in [n]$, as seen in \eqref{eq:alpha to t}.
\end{remark}

\begin{corollary} \label{corollary:Omega}
For a positive integer $1 \leq \bb \leq n-1$, $\Omega_{[\bb]}$ is irreducible.
\end{corollary}

\begin{proof}
By Proposition~\ref{proposition:propertyOmegaJ}, $\Omega_{[\bb]}$ is equidimensional of dimension $n-1-\bb$, so the top degree cohomology $H^{2(n-1-\bb)}(\Omega_{[\bb]})$ has a generator for each irreducible components of $\Omega_{[\bb]}$ (cf. \cite[\S~B.3, Lemma~4]{Ful97}).
However, since $H^{2(n-1-\bb)}(\Omega_{[\bb]}) \cong \Q$ by Proposition~\ref{proposition:condition(2)OmegaJ}, we conclude that $\Omega_{[\bb]}$ is irreducible. 
\end{proof}

We expect that for arbitrary $J \subset [n-1]$, $\Omega_J$ satisfies the conditions~(1) and (2) in Theorem~\ref{theorem:Cohomology1} when $I=[n-1]$, but this is not true in general.
The following gives an example which does not satisfy the condition (2) in Theorem~\ref{theorem:Cohomology1}.

\begin{example}[non-irreducible case] \label{ex:non-irreducible case} 
Let $n=4$ and we take $J=\{1,3\}$.
By Lemma~\ref{lemma:decomp_OmegaJ} we have 
\begin{align*}
\Omega_{\{1,3\}} = (\Omega_{\{1,3\}} \cap X_{\{1,3\}}^\circ) \coprod (\Omega_{\{1,3\}} \cap X_{\{1,2,3\}}^\circ).
\end{align*}
Recall that $X_{\{1,3\}}^\circ$ and $X_{\{1,2,3\}}^\circ$ are described as follows:
\begin{align*}
X_{\{1,3\}}^\circ &= \left\{ \left(
 \begin{array}{@{\,}cccc@{\,}}
x_1 & 1 & 0 & 0 \\
1 & 0 & 0 & 0 \\
0 & 0 & x_2 & 1 \\
0 & 0 & 1 & 0 
 \end{array}
 \right) \middle| \ x_1, x_2 \in \C \right\}, \\
X_{\{1,2,3\}}^\circ &= \left\{ \left(
 \begin{array}{@{\,}cccc@{\,}}
y_1 & y_2 & y_3 & 1 \\
y_2 & y_3 & 1 & 0 \\
y_3 & 1 & 0 & 0 \\
1 & 0 & 0 & 0 
 \end{array}
 \right) \middle| \ y_1, y_2, y_3 \in \C \right\}. 
\end{align*}
As explained in the proof of Lemma~\ref{lemma:Omegab}, the defining equations of $\Omega_{\{1,3\}}$ in $X_K^\circ \ (K=\{1,3\}, \{1,2,3\})$ is given by equations that the leading principal minors of orders $1$ and $3$ are zero.
Hence, we obtain 
\begin{align*}
\Omega_{\{1,3\}} \cap X_{\{1,3\}}^\circ &= \left\{ \left(
 \begin{array}{@{\,}cccc@{\,}}
0 & 1 & 0 & 0 \\
1 & 0 & 0 & 0 \\
0 & 0 & 0 & 1 \\
0 & 0 & 1 & 0 
 \end{array}
 \right) \right\}, \\
\Omega_{\{1,3\}} \cap X_{\{1,2,3\}}^\circ &= \left\{ \left(
 \begin{array}{@{\,}cccc@{\,}}
0 & y_2 & y_3 & 1 \\
y_2 & y_3 & 1 & 0 \\
y_3 & 1 & 0 & 0 \\
1 & 0 & 0 & 0 
 \end{array}
 \right) \middle| \ (2y_2-y_3^2)y_3=0 \right\}. 
\end{align*}
Note that $\Omega_{\{1,3\}} \cap X_{\{1,2,3\}}^\circ$ has two irreducible components.
Since $X_{\{1,2,3\}}^\circ$ is an open subset of $\Pet_4$, we see that $\Omega_{\{1,3\}}$ is \emph{not} irreducible.
As discussed in the proof of Corollary~\ref{corollary:Omega}, the top degree cohomology $H^2(\Omega_{\{1,3\}})$ is a two dimensional vactor space, which implies that $\Omega_{\{1,3\}}$ does not satisfy the condition (2) in Theorem~\ref{theorem:Cohomology1} when $I=[n-1]$.

We further analyze the topology of $\Omega_{\{1,3\}}$.
Let $Z_1$ be the irreducible component of $\Omega_{\{1,3\}} \cap X_{\{1,2,3\}}^\circ$ defined by $y_3=0$ and $Z_2$ the other irreducible component of $\Omega_{\{1,3\}} \cap X_{\{1,2,3\}}^\circ$ defined by $2y_2-y_3^2=0$.
Note that $Z_1$ is isomorphic to $\C$ with the coordinate $y_2$ and $Z_2$ is also isomorphic to $\C$ with the coordinate $y_3$. Their intersection $Z_1 \cap Z_2$ is equal to the point $\{w_{\{1,2,3\}}\}$.
If $y_2 \neq 0$ in $Z_1$, then a flag in $Z_1$ can be expressed as 
\begin{align*}
 \left(
 \begin{array}{@{\,}cccc@{\,}}
0 & y_2 & 0 & 1 \\
y_2 & 0 & 1 & 0 \\
0 & 1 & 0 & 0 \\
1 & 0 & 0 & 0 
 \end{array}
 \right)
 \equiv
  \left(
 \begin{array}{@{\,}cccc@{\,}}
0 & 1 & 0 & 0 \\
1 & 0 & 0 & 0 \\
0 & 1/y_2 & 0 & 1 \\
1/y_2 & 0 & 1 & 0 
 \end{array}
 \right) \ \ \ \textrm{in} \ \Fl(\C^4) \cong SL_4(\C)/B.
\end{align*}
Thus, the point at inﬁnity on $Z_1$ is $\{w_{\{1,3\}} \}$, which is exactly $\Omega_{\{1,3\}} \cap X_{\{1,3\}}^\circ$. 
Similarly, if $y_3 \neq 0$ in $Z_2$, then the point at inﬁnity on $Z_2$ is $\{w_{\{1,3\}} \}=\Omega_{\{1,3\}} \cap X_{\{1,3\}}^\circ$ by an elementary computation as follows:  
\begin{align*}
 \left(
 \begin{array}{@{\,}cccc@{\,}}
0 & y_3^2/2 & y_3 & 1 \\
y_3^2/2 & y_3 & 1 & 0 \\
y_3 & 1 & 0 & 0 \\
1 & 0 & 0 & 0 
 \end{array}
 \right)
 \equiv
  \left(
 \begin{array}{@{\,}cccc@{\,}}
0 & 1 & 0 & 0 \\
1 & 2/y_3 & 0 & 0 \\
2/y_3 & 2/y_3^2 & 0 & 1 \\
2/y_3^2 & 0 & 1 & 2/y_3 
 \end{array}
 \right) \xrightarrow[y_3 \rightarrow \infty]{} \left(
 \begin{array}{@{\,}cccc@{\,}}
0 & 1 & 0 & 0 \\
1 & 0 & 0 & 0 \\
0 & 0 & 0 & 1 \\
0 & 0 & 1 & 0 
 \end{array}
 \right) \ \ \ \textrm{in} \ \Fl(\C^4).
\end{align*}
Hence, $\Omega_{\{1,3\}}$ in $\Pet_4$ is homeomorphic to a topological space connecting two spheres at each of the south poles and each of the north poles (see Figure~\ref{picture:Topological picture for Omega1,3}).

\begin{figure}[h]
\begin{center}
\begin{picture}(200,110)
\put(100,95){\circle*{5}}
\put(100,15){\circle*{5}}
\qbezier(100,15)(80,55)(100,95)
%\qbezier(100,15)(30,55)(100,95)
\qbezier(100,15)(60,25)(60,55)
\qbezier(60,55)(60,85)(100,95)

\qbezier(100,15)(120,55)(100,95)
\qbezier(100,15)(140,25)(140,55)
\qbezier(140,55)(140,85)(100,95)

\qbezier(60,55)(75,45)(90,55)
%\qbezier(60,55)(75,65)(90,55)
\qbezier(60,55)(62,56)(64,57)
\qbezier(66,58)(68,59)(70,59)
\qbezier(73,59)(75,60)(77,59)
\qbezier(80,59)(82,59)(84,58)
\qbezier(86,57)(88,56)(90,55)

\qbezier(140,55)(125,45)(110,55)
\qbezier(140,55)(138,56)(136,57)
\qbezier(134,58)(132,59)(130,59)
\qbezier(127,59)(125,60)(123,59)
\qbezier(120,59)(118,59)(116,58)
\qbezier(114,57)(112,56)(110,55)

\put(95,0){$w_{\{1,2,3\}} \ (y_2=y_3=0, \ Z_1 \cap Z_2)$}
\put(95,105){$w_{\{1,3\}} \ (y_2\rightarrow\infty \ \textrm{in} \ Z_1, \ y_3\rightarrow\infty \ \textrm{in} \ Z_2)$}

\put(40,50){$Z_1$}
\put(150,50){$Z_2$}
\end{picture}
\end{center}
\caption{Topological picture for $\Omega_{\{1,3\}}$ in $\Pet_4$.}
\label{picture:Topological picture for Omega1,3}
\end{figure}
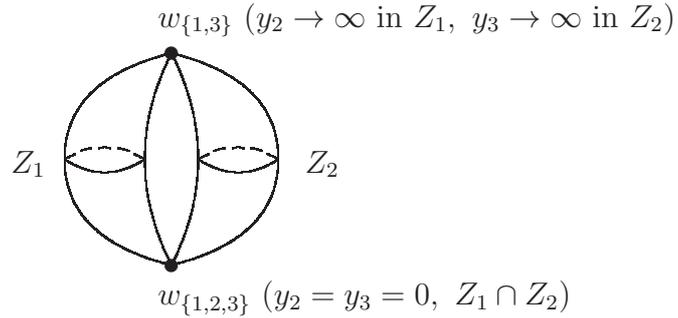

In particular, we see that $\Poin(\Omega_{\{1,3\}},\q)=1+\q+2\q^2$.
\end{example}

We see that $\Omega_{\{1,3\}}$ in $\Pet_4$ has two connected components in Example~\ref{ex:non-irreducible case}, which implies that $\Omega_J$ does not satisfy the condition (2) in Theorem~\ref{theorem:Cohomology1} when $I=[n-1]$, 
However, an irreducible component of $\Omega_{\{1,3\}}$ satisfies the desired conditions~(1) and (2) in Theorem~\ref{theorem:Cohomology1}. 
Therefore, we can naturally ask the following questions.

\begin{question} \label{question:OmegaJ} 
\begin{enumerate}
\item[(i)] Which of $\Omega_J$ is irreducible? (Is it true that $\Omega_J$ is irreducible if and only if $J$ is connected in the type $A_{n-1}$ Dynkin diagram?)
\item[(ii)] Is there an irreducible component $\tilde\Omega_J$ such that $\tilde\Omega_J$ satisfies the conditions~(1) and (2) in Theorem~\ref{theorem:Cohomology1} when $I=[n-1]$?
\end{enumerate}
\end{question}

\bigskip 

\section{Intersections of Peterson variety and Richardson varieties in type $A$} \label{section:Intersections of Peterson variety and Richardson varieties}

In this section we generalize results in Section~\ref{section:Intersections of Peterson variety and opposite Schubert varieties} to the intersections of Peterson veriety and some Richardson varieties in type $A_{n-1}$. 

We first recall \eqref{eq:intersectionXIOmegaJ}. For any two subsets $I,J \subset [n-1]$ with $I \supset J$, we define
\begin{align} \label{eq:XIJ}
X_I^J \coloneqq X_I \cap \Omega_J.
\end{align}
In other words, $X_I^J$ is the intersection between the Peterson variety $\Pet_n$ and the Richardson variety $X_{w_I} \cap \Omega_{w_J}$. Note that $X_I^J=X_I$ whenever $J=\emptyset$ and  $X_I^J=\Omega_J$ whenever $I=[n-1]$.

By Lemmas~\ref{lemma:condition(1)XI} and \ref{lemma:condition(1)OmegaJ}, one can see the following lemma.

\begin{lemma} \label{lemma:condition(1)XIJ}
Let $I$ and $J$ be two subsets of $[n-1]$ such that $I \supset J$.
The $S$-action on $\Pet_n$ preserves $X_I^J$ and the fixed point set is given by
\begin{align*}
(X_I^J)^S =\{w_K \mid J \subset K \subset I \}.
\end{align*}
\end{lemma}

Lemma~\ref{lemma:decomp_OmegaJ} and \eqref{eq:decomp XI} yields the following decomposition
\begin{align} \label{eq:decomp_XIJ}
X_I^J=\displaystyle\coprod_{K \supset J} \big( (\Omega_J \cap X_K^\circ) \cap X_I \big)=\displaystyle\coprod_{J \subset K \subset I} (\Omega_J \cap X_K^\circ).
\end{align}

For positive integers $\a,\bb$ with $\a \leq \bb$, we denote by $[\a,\bb]$ the set $\{\a,\a+1,\a+2,\ldots,\bb\}$. 
Here and below, we consider the case when $I=[\a,n-1]$ and $J=[\a,\bb]$ for some $\a \leq \bb$. 
Note that $X_{[\a,n-1]}^{[\a,\bb]}=\Omega_{[\bb]}$ whenever $\a=1$. 

\begin{lemma} \label{lemma:Xab}
Let $\a$ and $\bb$ be positive integers with $\a \leq \bb \leq n-1$. 
For a subset $K$ of $[n-1]$ such that $[\a,\bb] \subset K \subset [\a,n-1]$, one has
\begin{align*}
\Omega_{[\a,\bb]} \cap X_K^\circ \cong \C^{|K|-(\bb-\a+1)}.
\end{align*}
\end{lemma}

\begin{proof}
We prove this by a similar argument of the proof of Lemma~\ref{lemma:Omegab}. 
Since $[\a,\bb] \subset K \subset [\a,n-1]$, $K$ satisfies the following condition 
\begin{align*}
1, \ldots, \a-1 \not\in K \ \textrm{and} \ K \ \textrm{contains} \ [\a,\bb]. 
\end{align*}
Recalling that we may regard $X_K^\circ$ as the set of the block diagonal matrices in \eqref{eq:XwKcapPet}, the condition above means that the first $a-1$ diagonal entries of $g \in X_K^\circ$ are $1$, followed by $g_1 \in C^{(K_1)}$.
$$
g=
\begin{array}{rc|c|c|ccc|cl}
\hspace{-10pt}\ldelim({8}{1ex}[]&1    &           &    &           & & & &\rdelim){8}{1ex}[]  \\ \hline
							&    &  \ddots          &  &    &   & & &\\ \hline
							&        &     &  1  &           &   & &&\\ \hline
							&     &            &    &          &   &&&\\ 
							&     &            &    &           & g_1  &&&\\ 
							&     &            &    &           &   &&&\\ \hline
							&    &           &     &          &  &&\ddots& \\ 
							&\multicolumn{3}{c}{\raisebox{2ex}[1ex][0ex]{$\underbrace{\hspace{11ex}}_{{a-1}}$}}&&\multicolumn{3}{c}{\raisebox{2ex}[1ex][0ex]{}}&
\end{array}
$$
One can see that the equations that define $\Omega_{[\a,\bb]}=\cap_{j=\a}^{\bb} \Omega_{\{j\}}$ in $X_K^\circ \cong \C^{|K|}$ are given by equations that the leading principal minor of order $j$ of the matrix $g$ is zero for all $\a \leq j \leq \bb$ from \eqref{eq:defining equation Omegaj}.
In other words, the leading principal minor of order $j$ of the matrix $g_1$ is zero for all $1 \leq j \leq \bb-\a+1$.
By an argument similar to the proof of Lemma~\ref{lemma:Omegab}, we obtain $\Omega_{[\a,\bb]} \cap X_K^\circ \cong \C^{|K|-(\bb-\a+1)}$, as desired.
\end{proof}

Consider a filtration $X_{[\a,n-1]}^{[\a,\bb]}=Z_{n-\a} \supset Z_{n-\a-1} \supset \dots \supset Z_{\bb-\a+1} \supset Z_{\bb-\a}=\emptyset$ where $Z_k \coloneqq \displaystyle\bigcup_{[\a,\bb] \subset K \subset [\a,n-1] \atop |K| \leq k} (\Omega_{[\a,\bb]} \cap X_K)$ for each $\bb-\a+1 \leq k \leq n-\a$. (Note that $X_{[\a,n-1]}^{[\a,\bb]}=Z_{n-\a}$ by \eqref{eq:decomp XI} and \eqref{eq:decomp_XIJ}.)
Since each $Z_k \setminus Z_{k-1}$ is a disjoint union of $\Omega_{[\a,\bb]} \cap X_{K}^\circ$ with $|K|=k$ and $[\a,\bb] \subset K \subset [\a,n-1]$, Proposition~\ref{proposition:paving} and Lemma~\ref{lemma:Xab} yields the following proposition. 

\begin{proposition} \label{proposition:condition(2)XIJ}
For positive integers $\a \leq \bb \leq n-1$, the Poincar\'e polynomial of $X_{[\a,n-1]}^{[\a,\bb]}$ is equal to
\begin{align*}
\Poin(X_{[\a,n-1]}^{[\a,\bb]},\q)=(1+\q^2)^{n-1-b}.
\end{align*}
\end{proposition}

Lemma~\ref{lemma:condition(1)XIJ} and Proposition~\ref{proposition:condition(2)XIJ} correspond to the conditions~(1) and (2) in Theorem~\ref{theorem:Cohomology1} when $I=[\a,n-1]$ and $J=[\a,\bb]$. 
Therefore, we obtain an explicit presentation of the cohomology ring of $X_{[\a,n-1]}^{[\a,\bb]}$.

\begin{theorem} \label{theorem:mainXIJ}
Let $\a$ and $\bb$ be positive integers such that $1 \leq \a \leq \bb \leq n-1$. 
Let $X_{[\a,n-1]}^{[\a,\bb]}$ be the intersection of the Peterson variety $\Pet_n$ and the Richardson variety defined in \eqref{eq:XIJ}. 
Then, the restriction map $H^*(\Pet_n) \rightarrow H^*(X_{[\a,n-1]}^{[\a,\bb]})$ is surjective. 
Also, there is an isomorphism of graded $\Q$-algebras
\begin{align*} 
H^*(X_{[\a,n-1]}^{[\a,\bb]}) \cong \RR/(\alpha_k \varpi_k \mid \bb+1 \leq k \leq n-1) +(\alpha_j \mid \a \leq j \leq \bb)+(\varpi_i \mid 1 \leq i \leq \a-1),
\end{align*}
which sends $\alpha$ to $c_1(L_\alpha^*)$. 
\end{theorem}

\begin{remark}
We can also give an explicit presentation of the $S$-equivariant cohomology ring $H^*_S(X_{[\a,n-1]}^{[\a,\bb]})$ from Theorem~\ref{theorem:psiIJiso}.
\end{remark}

We asked natural questions in Section~\ref{section:Intersections of Peterson variety and opposite Schubert varieties}. The same questions can be posed in the general setting.

\begin{question} \label{question:XIJ}
\begin{enumerate}
\item[(i)] Which of $X_I^J$ is irreducible? 
\item[(ii)] Does some irreducible component of $X_I^J$ satisfy the conditions~(1) and (2) in Theorem~\ref{theorem:Cohomology1}?
\end{enumerate}
\end{question}

\end{document}